\documentclass[11pt]{amsart}
\usepackage{fullpage}
\usepackage{color}
\usepackage{graphicx,psfrag} 
\usepackage{color} 
\usepackage{tikz}
\usepackage{pgffor}
\usepackage{hyperref}
\usepackage{todonotes}
\usepackage{eqnarray,amsmath}

\usepackage{perpage}     

\usepackage{wrapfig}
\usepackage{enumerate}
\usepackage{subfigure}
\usepackage{caption}

\usepackage[normalem]{ulem}
\usepackage[makeroom]{cancel} 
\usepackage{amssymb} 

\MakePerPage{footnote}   

\newtheorem{theorem}{Theorem}[section]
\newtheorem{lemma}[theorem]{Lemma}

\newtheorem{proposition}[theorem]{Proposition}

\theoremstyle{definition}
\newtheorem{definition}[theorem]{Definition}
\newtheorem{example}[theorem]{Example}
\theoremstyle{remark}
\newtheorem{remark}[theorem]{Remark}

\newtheorem{notation}[theorem]{Notation}

\newtheorem{property}[theorem]{Property}
\newtheorem{question}[theorem]{Question}

\newtheorem{ex-prop}[theorem]{Example-Proposition}

\numberwithin{equation}{section}

\definecolor{gray}{rgb}{.5,.5,.5}

\definecolor{black}{rgb}{0,0,0}

\definecolor{blue}{rgb}{0,0,1}

\definecolor{red}{rgb}{1,0,0}

\definecolor{green}{rgb}{0,1,0}

\definecolor{yellow}{rgb}{1,1,.4}

\definecolor{purple}{rgb}{1,0,1}

\definecolor{gold}{rgb}{.5,.5,.2}

\definecolor{darkgreen}{rgb}{0,.5,0}

\definecolor{greenbean}{RGB}{199, 237, 204}

\definecolor{RED}{rgb}{1,0,0}

\begin{document}

\title{The Jones polynomials of 3-bridge knots\\ via Chebyshev knots and billiard table diagrams}

\author{Moshe Cohen}
\address{Department of Mathematics, Technion -- Israel Institute of Technology, Haifa 32000, Israel}
\email{mcohen@tx.technion.ac.il}


\begin{abstract}
This work presents formulas for the Kauffman bracket and Jones polynomials of 3-bridge knots using the structure of Chebyshev knots and their billiard table diagrams.  In particular, these give far fewer terms than in the Skein relation expansion.  The subject is introduced by considering the easier case of 2-bridge knots, where some geometric interpretation is provided, as well, via combinatorial tiling problems.
\end{abstract}


\maketitle

\section{Introduction}
\label{sec:Intro}

The Kauffman bracket and Jones polynomials have been well-studied for the class of 2-bridge or rational knots \cite{DuzShk, GuJo, LeeLeeSeo, Kan1, Kan2, Lew:dis, LicMil, LuZho, Nak1, Nak2, QYAQ, Sto} (see also \cite{Bir:3braid}).  The present work uses a new model to study these polynomials on 2-bridge knots that can naturally be extended to 3-bridge knots, the main goal here.

A (long) knot is a \emph{Chebyshev knot} $T(a,b,c)$ if it admits a one-to-one parametrization of the form $x = T_a(t)$; $y = T_b(t)$; $z = T_c(t + \varphi)$, where $T_n(\cos t')=\cos(nt')$ is the $n$-th Chebyshev polynomial of the first kind, $t\in\mathbb{R}$, $a, b, c\in\mathbb{Z}$, and $\varphi$ is a real constant.

These are generalizations of \emph{harmonic knots} when the third coordinate has no phase shift and when the three parameters are pairwise coprime integers \cite{Comstock, KosPec4, KosPec3, KosPec:Harm}.  
  These are polynomial analogues of the famous Lissajous knots \cite{BDHZ, BHJS, Crom, HosZir, JonPrz, Lam, Lam:dis}.

\renewcommand{\thesubfigure}{}
\begin{figure}[htbp]
\centering
\subfigure[A.    $T(3,7)$.]{
\begin{tikzpicture}[scale=.6]
    \foreach \i in {0,...,7} {
        \draw [very thin,lightgray] (\i,1) -- (\i,4);
    }
    \foreach \i in {1,...,4} {
        \draw [very thin,lightgray] (0,\i) -- (7,\i);
    }
\draw[thick](-.25,.75) -- (3,4) -- (6,1) -- (7,2) -- (5,4) -- (2,1) -- (0,3) -- (1,4) -- (4,1) -- (7,4) -- (7.25,4.25);
\draw[color=white](-.25,-.25) -- (0,0);
\draw[color=white](7,5) -- (7.25,5.25);
\end{tikzpicture}
}
\hspace{2em}
\subfigure[B.    $T(5,7)$.]{
\begin{tikzpicture}[scale=.6]
    \foreach \i in {0,...,7} {
        \draw [very thin,lightgray] (\i,0) -- (\i,5);
    }
    \foreach \i in {0,...,5} {
        \draw [very thin,lightgray] (0,\i) -- (7,\i);
    }
\draw[thick](-.25,-.25) -- (5,5) -- (7,3) -- (4,0) -- (0,4) -- (1,5) -- (6,0) -- (7,1) -- (3,5) -- (0,2) -- (2,0) -- (7.25,5.25);
\end{tikzpicture}
}
\caption{\label{fig:Cheb} Chebyshev knots yielding billiard table diagrams $T(a,b)$.}
\end{figure}
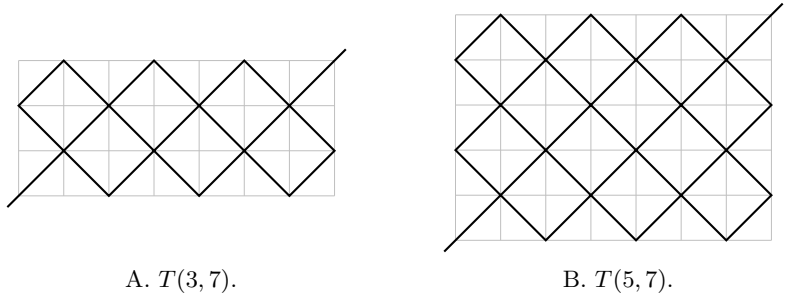

Chebyshev knots and links $T(a,b,c)$ can be projected to $(a-1)\times (b-1)$ billiard table diagrams.  Below we consider both positive and negative crossings at once, and so we omit the third variable, viewing the diagrams without crossings as in Figure \ref{fig:Cheb}, showing $T(3,7)$ and $T(5,7)$.   

Some of these are examples of periodic trajectories on rational polygonal billiards; such billiards can be unfolded into translation surfaces.  Periodic trajectories are studied more generally in dynamical systems and ergodic theory \cite{MasTab:bill, Masur:erg}.

Chebyshev knots appear in several works by Koseleff and Pecker, who in particular show:
\begin{proposition}
\label{thm:bridge}
\cite[Proposition 9 specialized to $m=3$]{KosPec3}
Every 3-bridge knot $K$ 
 has a projection which is a Chebyshev curve $x = T_5(t)$, $y =T_b(t)$, where $b\equiv 2$ (mod $10$).
\end{proposition}
They use the general case of Proposition 9 to show their main Theorem 3 that every knot has a projection that is a Chebyshev plane curve.  
They study Chebyshev diagrams for 2-bridge knots in \cite{KosPec4}, leading them to connections with Fibonacci polynomials when considering the Alexander polynomial in \cite{PVK:fib}.

This is far from surprising based on the combinatorial literature.  In \cite{CoDaRu}, the author with Dasbach and Russell describes the determinant expansion of the Alexander polynomial of a knot in terms of perfect matchings of its \emph{balanced overlaid Tait graph} or so-called ``dimer graph.''  When this graph is a subgraph of a grid graph, this setting can be translated into the problem of domino tilings of a game board.  The problem of counting domino tilings of a $2\times n$ game board is famously known to give the Fibonacci numbers.

In previous work by the author \cite{Co:clock} with Teicher, it is remarked that the balanced overlaid Tait graph of a Chebyshev knot's billiard table diagram is in fact a grid graph.  These objects appear often in combinatorial literature, and so there is hope that using this model gives access to results that will not hold for general knot diagrams.

Domino tilings of rectangles have also been well-studied (\cite{KlaPol, Rea, Sta:rect, BHS} to name a few).   For a more detailed review of the combinatorial background material, see the text on Matching Theory by Lov{\'a}sz and Plummer \cite{LovPlum} and lecture notes on dimers by Kenyon \cite{Ken}.

\medskip

\textbf{Results.  }  The tuple Notation \ref{note:tuple} allows the computation of invariants for whole classes of Chebyshev knots without specifying crossing information, as each term in the tuple corresponds with a crossing, ordered from bottom to top in each column and then from left to right on the columns.  The term $A^\pm$ in the $i$-th position, for example, instructs one to include in the product the term $A$ (or $A^{-1}$) if the $i$-th crossing is positive (or negative, respectively).

\medskip

\noindent
\textbf{Theorem \ref{thm:main}.}
\textit{
(Main Theorem) Following Notation \ref{note:thm} below, the Kauffman bracket polynomials of $T(5,b)$ (for indeterminate crossing information), including the class of 3-bridge knots, are
\begin{align*}
h_b= {} & \sum_{i=3}^{b-1}([h_3,P_{i-2}]+[h_2,P_{i-1}]+[Q_i],\mathcal{P}_{b-1-i})
\end{align*}
for $b\geq 4$ with base cases $h_1=1$, $h_2=(A^\pm,A^\pm)+\delta(A^\pm,A^\mp)+\delta(A^\mp,A^\pm)+\delta^2(A^\mp,A^\mp)$,  and $h_3=(h_2,A^\pm,A^\pm)+(f_2^\mp,f_2^\mp,A^\mp,A^\mp)+(g_2,A^\pm,A^\mp)+(f_2^\mp,f_2^\pm,A^\mp,A^\pm)$, where $\delta=(-A^2-A^{-2})$.
}

\medskip

\textbf{Organization. }  This paper is organized as follows:  Section \ref{sec:Background} contains some brief background material on the Jones polynomial.  Section \ref{sec:Fibonacci} considers the $a=3$ case, yielding 2-bridge knots, and its connections with Fibonacci numbers.  Section \ref{sec:4} considers the $a=4$ case in brief so that some of its results can be applied to the next case.  Section \ref{sec:5} considers the $a=5$ case with 3-bridge knots, giving small examples before proving the main results in Section \ref{subsec:results}.

For 2-bridge knots, note Theorem \ref{thm:3rec} and Subsection \ref{subsec:3geom}, where the result is simplified geometrically.  Proposition \ref{prop:OEIS}
 counts Padovan number $P(b-4)$ terms in this expansion, even fewer than the Fibonacci number $F(b-1)$ terms in the spanning tree model!

For 3-bridge knots, note the Main Theorem \ref{thm:main} above.  Proposition \ref{prop:exponential} counts $2^{b-4}$ terms in the sum over all $\mathcal{P}_j$, 
 far fewer than expected for $2(b-1)$ crossings!

Recursive formulas for the writhe are given in Section \ref{sec:writhe} to obtain Jones polynomials, although perhaps it is enough to have the list of coefficients of the Kauffman bracket polynomials.

\medskip

\textbf{Acknowledgements. }  This work arose based explicitly on three week-long visits with Pierre-Vincent Koseleff:  in June 2012 and June 2013 at the Universit\'e Pierre \& Marie Curie (Paris 6) funded by INRIA-Rocquencourt Ouragan and ANR 
 Structures G\'{e}om\'{e}triques Triangul\'{e}es
; and in May 2014 at the Technion, when he spoke about \cite{PVK:fib}.  Thanks also go to Misha Polyak.

\medskip

We explain here the notation used in the Main Theorem, although these ideas are developed more slowly in Sections \ref{sec:Fibonacci}, \ref{sec:4}, and \ref{sec:5}.

\noindent
\textbf{Notation \ref{note:thm}.}
Let $\square^i$ be $[\square,\square,\ldots,\square]$ written $i$ times.  Let $f_2^\pm=-A^{\mp3}$ and $g_2$ be the Kauffman bracket polynomials of $T(3,2)$ and $T(4,2)$, respectively.  Let $X=\delta[A^\pm,A^\pm]+[A^\pm,A^\mp]+[A^\mp,A^\pm]$.

Let $P'_1=\widetilde{P}'_1=[A^\pm,A^\pm]$, and the remaining $P'_i$ and $\widetilde{P}'_i$ be as follows: \\
$P'_2=[f_2^\mp,f_2^\mp,A^\mp,A^\mp]+[A^\pm,f_2^\mp,A^\pm,A^\mp]+\delta[A^\pm,A^\pm,A^\mp,A^\pm]+[A^\pm,A^\mp,A^\mp,A^\pm]+[A^\mp,A^\pm,A^\mp,A^\pm]$,\\
$\widetilde{P}'_2=[f_2^\mp,f_2^\mp,A^\mp,A^\mp] +[f_2^\mp,A^\pm,A^\mp,A^\pm] +\delta[A^\pm,A^\pm,A^\pm,A^\mp]+[A^\mp,A^\pm,A^\pm,A^\mp]+[A^\pm,A^\mp,A^\pm,A^\mp]$, 

\begin{equation*}
P'_i=
\begin{cases}
[f_2^\mp,A^\pm,A^\mp,A^\mp,[f_2^\mp,A^\mp,A^\mp,A^\mp]^{j},A^\pm,A^\mp] & \\
\mbox{\hspace{5mm}} + [A^\pm,[f_2^\mp,f_2^\mp,A^\mp,A^\mp]^{j+1},A^\pm] &\mbox{ for $i=2j+3$ odd and }\\
[A^\pm,[f_2^\mp,f_2^\mp,A^\mp,A^\mp]^{j},f_2^\mp,A^\pm,A^\mp] & \\
\mbox{\hspace{5mm}} + [X,[f_2^\mp,A^\mp,A^\mp,A^\mp]^{j},A^\mp,A^\pm] &\mbox{ for $i=2j+2$ even. }
\end{cases}
\end{equation*}

\begin{equation*}
\widetilde{P}'_i=
\begin{cases}
[f_2^\mp,A^\pm,A^\mp,[f_2^\mp,f_2^\mp,A^\mp,A^\mp]^{j},f_2^\mp,A^\pm,A^\mp] & \\
\mbox{\hspace{5mm}} + [A^\pm,f_2^\mp,A^\mp,A^\mp,[A^\mp,f_2^\mp,A^\mp,A^\mp]^{j},A^\mp,A^\pm] &\mbox{ for $i=2j+3$ odd and }\\
[f_2^\mp,A^\pm,A^\mp,[f_2^\mp,f_2^\mp,A^\mp,A^\mp]^{j},A^\pm] & \\
\mbox{\hspace{5mm}} + [X,[f_2^\mp,A^\mp,A^\mp,A^\mp]^{j},A^\pm,A^\mp] &\mbox{ for $i=2j+2$ even. }
\end{cases}
\end{equation*}

The blocks $P'_i$ and $\widetilde{P}'_i$ correspond to elements of size $i$ in $\mathcal{P}_n$, the set of all partitions of the integer $n$.  
  Let $j$ be the position in $n$ of the first part of one of these elements.  For convenience of notation we set 
\begin{equation*}
P_i:=P_i(j)=
\begin{cases}
P'_i  &\mbox{ if $i+j$ is odd and }\\
\widetilde{P}'_i  &\mbox{ if $i+j$ is even. }
\end{cases}
\end{equation*}

For $i\geq 3$, let $Q_i$ be as follows:
\begin{equation*}
Q_i=
\begin{cases}
[f_2^\mp,f_2^\pm,A^\mp,[f_2^\mp,f_2^\mp,A^\mp,A^\mp]^{j},f_2^\mp,A^\pm,A^\mp] &\\
\mbox{\hspace{5mm}} +[f_2^\pm,f_2^\mp,A^\mp,A^\mp,[A^\mp,f_2^\mp,A^\mp,A^\mp]^{j},A^\mp,A^\pm] &\mbox{ for $i=2j+3$ odd and }\\
[f_2^\mp,f_2^\pm,A^\mp,[f_2^\mp,f_2^\mp,A^\mp,A^\mp]^{j},A^\pm] & \\
\mbox{\hspace{5mm}} + [g_2,[f_2^\mp,A^\mp,A^\mp,A^\mp]^{j},A^\pm,A^\mp] &\mbox{ for $i=2j+2$ even. }
\end{cases}
\end{equation*}

\section{The Jones polynomial}
\label{sec:Background}

For the sake of completion, we provide the necessary definitions due to Kauffman \cite{KauffmanBracket}.  Note also that Thistlethwaite \cite{Th} gives a spanning tree expansion; Dasbach, Futer, Kalfagianni, Lin, and Stoltzfus \cite{DaFuKaLiSt} compute these via associated ribbon graphs; and the author computes these of pretzel knots in \cite{Co:jones} as matrix determinants.

Given an unoriented crossing depicted locally in a link diagram $L$, let $L_0$ and $L_\infty$ be the two smoothings, also called the $A$- and the $B$-smoothings, as in Figure \ref{fig:Jones}A.

\renewcommand{\thesubfigure}{}
\begin{figure}[htbp]
\centering
\subfigure[A.  An unoriented crossing and the two smoothings.]{
\begin{tikzpicture}
\draw[loosely dashed] (.5,.5) circle (.7cm);
\draw[-] (0,0) -- (1,1);
\draw (1,0) -- (.55,.45);
\draw[-] (0,1) -- (.45,.55);
\draw[loosely dashed] (2.5,.5) circle (.7cm);
\draw[-] (2,0) arc (-45:45:.7cm);
\draw[-] (3,1) arc (135:225:.7cm);
\draw[loosely dashed] (4.5,.5) circle (.7cm);
\draw[-] (5,0) arc (45:135:.7cm);
\draw[-] (4,1) arc (225:315:.7cm);
\end{tikzpicture}
}
\hspace{2em}
\subfigure[B.  Positive and negative crossings contribute $+1$ and $-1$ to the writhe, respectively.]{
\begin{tikzpicture}
\draw[loosely dashed] (.5,.5) circle (.7cm);
\draw[->] (0,0) -- (1,1);
\draw (1,0) -- (.55,.45);
\draw[<-] (0,1) -- (.45,.55);
\draw (-1.5,.5) node {positive};
\draw (4.5,.5) node {negative};
\draw[loosely dashed] (2.5,.5) circle (.7cm);
\draw[<-] (2,1) -- (3,0);
\draw (2,0) -- (2.45,.45);
\draw[<-] (3,1) -- (2.55,.55);
\end{tikzpicture}
}
\caption{\label{fig:Jones} Resolutions and writhe needed to define the Jones polynomial.}
\end{figure}
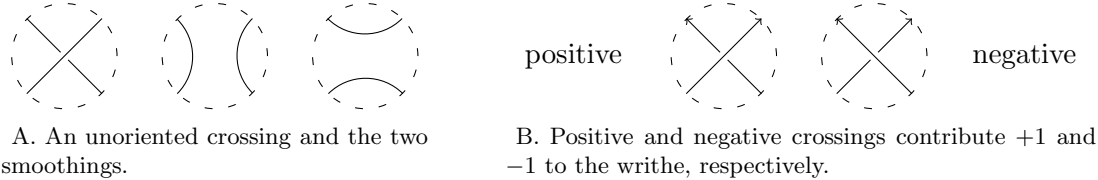

The \emph{Kauffman bracket polynomial $\langle L\rangle$} 
 of a link $L$ can be defined by
\begin{enumerate}
	\item Smoothing relation:  $\langle L \rangle =A \langle L_0 \rangle + A^{-1} \langle L_\infty \rangle$
	\item Stabilization:  $\langle U\sqcup L\rangle=\delta \langle L\rangle$
	\item Normalization:  $\langle U\rangle=1$
\end{enumerate}
where $U$ is the unknot, $\sqcup$ is the disjoint union, and $\delta=(-A^2-A^{-2})$.

The \emph{writhe} $w(D)$ of an oriented diagram is the sum over all crossings of the evaluation $+1$ for positive crossings and $-1$ for negative crossings as in Figure \ref{fig:Jones}B.

The \emph{Jones polynomial $V_L(t)$} of a link $L$ given an oriented diagram $D$ can be defined via the Kauffman bracket polynomial (along with the substitution $A=t^{-1/4}$) by
$$V_L(t)=(-A^{-3})^{w(D)}\langle L \rangle.$$

Note that the Kauffman bracket polynomial is a specialization of the \emph{signed Tutte polynomial} on the signed plane graph obtained from a knot diagram by its checkerboard coloring.

\section{The height-3 case and Fibonacci polynomials}
\label{sec:Fibonacci}

In this section we consider Chebyshev knots $T(3,b)$, including the class of 2-bridge knots, and their Kauffman bracket polynomials, denoted $f_b$.

\subsection{Small examples}
\label{subsec:3ex}
We begin with some base cases to introduce some notation and ideas that will be useful in the main Theorem \ref{thm:3rec} of this section below.

The reader familiar with the Jones polynomial may want to skip to this Theorem, but pay attention to Notation \ref{note:tuple} and the notation used in the end of Example \ref{ex:T35}.

\begin{example}
\label{ex:T31-2}
The knot $T(3,1,\emptyset)$ is the unknot and thus has Jones polynonial equal to 1.  The Kauffman bracket polynomial, which we denote by $f_1$ is also equal to 1.

Both knots $T(3,2,\pm)$ are the unknot with a single Reidemeister I move and thus have Jones polynonial equal to 1.  The Kauffman bracket polynomial, which we denote by $f_2^\pm$, depending on whether the crossing is positive or negative, respectively, is $A+A^{-1}\delta$ or $A^{-1}+A\delta$, and so:
\begin{equation}
\label{eq:f2}
f_2^\pm=-A^{\mp3}.
\end{equation}
\end{example}

\begin{notation}
\label{note:tuple}
Beginning in the next example, we introduce the following $n$-tuple notation for $T(3,n+1,c)$ knots with $n$ crossings so that we may consider them without determining the signs:  the $i$-th coordinate of the tuple gives the factor determined by the sign of the $i$-th crossing, as ordered from left to right.  These products are then summed over all possible resolutions to give the Kauffman bracket polynomial.
\end{notation}

We use this ordering on the crossings to perform resolutions starting from the end.

\begin{example}
\label{ex:T33}
Consider the link $T(3,3,\pm\pm)$, where the two crossings are independent of each other.  The Kauffman bracket polynomial is
\begin{equation}
\label{eq:f3}
f_3=(f_2^\pm,A^\pm)+(f_2^\mp,A^\mp).
\end{equation}
We note, for the sake of the recursion that will follow, that we get $(\cdot,\cdot,A^\pm)+(\cdot,f_2^\mp,A^\mp)$ in general.

Since $T(3,3,++)=T(3,3,--)$ is the unlink of two unknotted components and $T(3,3,+-)=T(3,3,-+)$ in the Hopf link, the four possibilities give only two.  These are, respectively,
$$-A^{-3}A^{}-A^{3}A^{-1}=-A^{-2}-A^{2}=\delta \quad \text{ and } \quad -A^{-3}A^{-1}-A^{3}A^{}=-A^{-4}-A^{4}.$$
\end{example}

\begin{example}
\label{ex:T34}
Consider the link $T(3,4,\pm\pm\pm)$, where the three crossings are independent of each other.  The Kauffman bracket polynomial is
\begin{equation}
\label{eq:f4}
f_4=(f_3,A^\pm)+(f_2^\pm,f_2^\mp,A^\mp)=(f_2^\pm,A^\pm,A^\pm)+(f_2^\mp,A^\mp,A^\pm)+(f_2^\pm,f_2^\mp,A^\mp).
\end{equation}
The link $T(3,4,+-+)$ is the trefoil.  The Kauffman bracket polynomial, denoted by $f_4^{+-+}$, is 
$$Af_3^{+-}+A^{-1}(f_2^+)f_2^+=(-A^4-A^{-4})A+(-A^{-3})(-A^{-3})A^{-1}=-A^5-A^{-3}+A^{-7}.$$
  With a writhe of $3$, this gives a Jones polynomial of $A^{-4}+A^{-12}-A^{-12}=t+t^3-t^4$, as confirmed by the Knot Atlas \cite{knotatlas}.
\end{example}

\begin{example}
\label{ex:T35}
Consider the link $T(3,5,\pm\pm\pm\pm)$, where the four crossings are independent of each other.  The Kauffman bracket polynomial is
\begin{equation}
\label{eq:f5}
f_5=(f_4,A^\pm)+(f_3,f_2^\mp,A^\mp)
\end{equation}
\begin{align*}
&= (f_2^\pm,A^\pm,A^\pm,A^\pm)+(f_2^\mp,A^\mp,A^\pm,A^\pm)+(f_2^\pm,f_2^\mp,A^\mp,A^\pm)+(f_2^\pm,A^\pm,f_2^\mp,A^\mp)+(f_2^\mp,A^\mp,f_2^\mp,A^\mp) \\
&= (f_2^\pm,A^\pm,A^\pm,A^\pm)+(f_2^\pm,A^\pm,f_2^\mp,A^\mp)+(f_2^\mp,A^\mp,A^\pm,A^\pm)+(f_2^\mp,A^\mp,f_2^\mp,A^\mp)+(f_2^\pm,f_2^\mp,A^\mp,A^\pm) \\
&= ([f_2^\pm,A^\pm]+[f_2^\mp,A^\mp],[A^\pm,A^\pm]+[f_2^\mp,A^\mp])+(f_2^\pm,f_2^\mp,A^\mp,A^\pm) \\
&= (f_3,C)+(f_2^\pm,f_2^\mp,A^\mp,A^\pm),
\end{align*}
where we set $C=[A^\pm,A^\pm]+[f_2^\mp,A^\mp]$ and from Example \ref{ex:T33} above $f_3=[f_2^\pm,A^\pm]+[f_2^\mp,A^\mp]$.
\end{example}

We now have access to the notation that will make the following main Theorem of this section meaningful.

\begin{theorem}
\label{thm:3rec}
Recall that $C=[A^\pm,A^\pm]+[f_2^\mp,A^\mp]$ and substitute $C$ starting from the left as often as possible. The Kauffman bracket polynomials $f_b$ of $T(3,b)$, including the class of 2-bridge knots, obey the following recursion rules.  If a summand in $f_{b-1}$ ends in:
\begin{itemize}
	\item[($A$)] $(...,A^\pm)$, then it appears as a summand in $f_b$ with a $C$ replacing the $A^\pm$.
	\item[($f$)] $(...,[f_2^\mp,A^\mp])$, then it appears as a summand in $f_b$ ending with $A^\pm$.
	\item[($C$)] $(...,C)$, then it appears as two summands in $f_b$ with $[C,A^\pm]+[A^\pm,f_2^\mp,A^\mp]$ replacing the $C$.
\end{itemize}
In particular, here is a list of Kauffman brackets for this class:
\begin{align*}
f_3= {} & (f_2^\pm,A^\pm)+(f_2^\mp,A^\mp) \\
f_4= {} & (f_3,A^\pm)+(f_2^\pm,f_2^\mp,A^\mp) \\
f_5= {} & (f_3,C)+(f_2^\pm,f_2^\mp,A^\mp,A^\pm) \\
f_6= {} & (f_3,[C,A^\pm]+[A^\pm,f_2^\mp,A^\mp])+(f_2^\pm,f_2^\mp,A^\mp,C) \\
f_7= {} & (f_3,C,C)+(f_3,A^\pm,f_2^\mp,A^\mp,A^\pm)+(f_2^\pm,f_2^\mp,A^\mp,[C,A^\pm]+[A^\pm,f_2^\mp,A^\mp]) \\
f_8= {} & (f_3,C,[C,A^\pm]+[A^\pm,f_2^\mp,A^\mp])+(f_3,A^\pm,f_2^\mp,A^\mp,C) +(f_2^\pm,f_2^\mp,A^\mp,C,C)+(f_2^\pm,f_2^\mp,A^\mp,A^\pm,f_2^\mp,A^\mp,A^\pm)\\
f_9= {} & (f_3,C,C,C)+(f_3,C,A^\pm,f_2^\mp,A^\mp,A^\pm)+(f_3,A^\pm,f_2^\mp,A^\mp,[C,A^\pm]+[A^\pm,f_2^\mp,A^\mp]) \\
& +(f_2^\pm,f_2^\mp,A^\mp,C,[C,A^\pm]+[A^\pm,f_2^\mp,A^\mp]) +(f_2^\pm,f_2^\mp,A^\mp,A^\pm,f_2^\mp,A^\mp,C)\\
f_{10}= {} & (f_3,C,C,[C,A^\pm]+[A^\pm,f_2^\mp,A^\mp])+(f_3,C,A^\pm,f_2^\mp,A^\mp,C) +(f_3,A^\pm,f_2^\mp,A^\mp,C,C)\\
& +(f_3,A^\pm,f_2^\mp,A^\mp,A^\pm,f_2^\mp,A^\mp,A^\pm) +(f_2^\pm,f_2^\mp,A^\mp,C,C,C) +(f_2^\pm,f_2^\mp,A^\mp,C,A^\pm,f_2^\mp,A^\mp,A^\pm)\\
& +(f_2^\pm,f_2^\mp,A^\mp,A^\pm,f_2^\mp,A^\mp,[C,A^\pm]+[A^\pm,f_2^\mp,A^\mp]) 
\end{align*}
\end{theorem}

\begin{remark}
\label{rem:implications}
Implications of this recursion include:
\begin{itemize}
	\item This sum has far fewer terms than the $2^{b-1}$ terms using the Skein relation for the Kauffman bracket polynomial.
	\item This sum has fewer terms than the $F(b-1)$ (Fibonacci number) terms in the domino tiling interpretation given in the Subsection \ref{subsec:3geom}, corresponding to the spanning tree expansion of the Jones polynomial, as in \cite{Co:jones}.
	\item This recursion is of only one level, unlike the Fibonacci recursion.
\end{itemize}
\end{remark}

\begin{proof}
Observe that $(C)$ follows from $(A)$ and $(f)$:  a term in $h_b$ ending in $C$ is really two terms, one ending in $[A^\pm,A^\pm]$ and the other ending in $[f_2^\mp,A^\mp]$.  By $(A)$ and $(f)$ these produce two new terms ending in $[A^\pm,C]$ and $[f_2^\mp,A^\mp,A^\pm]$; however these are really three terms:  ending in $[A^\pm,A^\pm,A^\pm]$, $[A^\pm,f_2^\mp,A^\mp]$, and $[f_2^\mp,A^\mp,A^\pm]$.  These can be regrouped to achieve the result in $(C)$.

We prove $(A)$ and $(f)$ by induction.  The base cases are handled in Examples 
 \ref{ex:T34} and \ref{ex:T35}.  Suppose that the induction hypothesis holds up to $b$, and consider the case $f_{b+1}$.

The final $b$-th crossing can be resolved in two ways:  vertically, yielding a summand ending in $A^\pm$, or horizontally, yielding a summand ending in $A^\mp$ and forcing the ($b-1$)-st term in this summand to be $f_2^\mp$.  This gives, as was noted above, $(\cdot,\cdot,A^\pm)+(\cdot,f_2^\mp,A^\mp)$ in general.  
  Observe that the $[\cdot,\cdot]$ and the $[\cdot]$ must represent terms in $f_b$ and $f_{b-1}$, respectively, as after the smoothings, one obtains smaller billiard table diagrams of these sizes.

Repeat this resolution for the ($b-1$)-st crossing of the first term, obtaining $(\cdot,\cdot,\cdot,A^\pm,A^\pm) +(\cdot,\cdot,f_2^\mp,A^\mp,A^\pm)$.  A similar observation as above holds for these terms with the appropriate index.

Amongst these three terms there is a single term in $f_b$ ending in $[f_2^\mp,A^\mp]$.  Clearly this term is a term in $f_{b+1}$ ending in $[f_2^\mp,A^\mp,A^\pm]$, proving $(f)$.

The other two terms can be combined to yield $(\cdot,\cdot,C)$.  Since these terms were terms in $f_b$ ending in $A^\pm$, this proves $(A)$.
\end{proof}

\begin{proposition}
\label{prop:OEIS}
The number of terms in the expansion of the Kauffman bracket polynomial $f_b$ of $T(3,b)$ of Theorem \ref{thm:3rec} gives a sequence that is an offset by four of the Padovan sequence: $a(n) = a(n-2) + a(n-3)$ with $a(0)=1$, $a(1)=a(2)=0$, which is A000931 in The On-Line Encyclopedia of Integer Sequences \cite{OEIS931}.
\end{proposition}

For information about the Padovan sequence (and how it relates to the Fibonacci sequence), see, for example, \cite{deW}.  For more connections between these sequences and grid graphs, see \cite{Eul:grid}.

\begin{proof}
Write as a triple $(x_i,y_i,z_i)$ the numbers of terms ending in $A^\pm$, $[f_2^\mp,A^\mp]$, and $C$, respectively, of $f_i$, the Kauffman bracket polynomial of $T(3,i)$, such that $f_i$ has $x_i+y_i+z_i$ terms in total.

The first two bases cases are $f_4$ with tuple (1,1,0) and $f_5$ with tuple (1,0,1), as seen in Examples \ref{ex:T34} and \ref{ex:T35}, respectively.  One can check that the next base case $f_6$ has tuple (1,1,1).  Suppose that the induction hypothesis holds up to $b$, and consider the case $f_{b+3}$.

Following Theorem \ref{thm:3rec}: the tuple associated to $f_{b}$ is $(y_b,z_b,x_b)$ for a total of $x_b+y_b+2z_b$ terms;  the tuple associated to $f_{b+1}$ is $(y_b+z_b,z_b,x_b)$ for a total of $x_b+y_b+2z_b$ terms; the tuple associated to $f_{b+2}$ is $(x_b+z_b,x_b,y_b+z_b)$ for a total of $2x_b+y_b+2z_b$ terms; and the tuple associated to $f_{b+3}$ is $(x_b+y_b+z_b,y_b+z_b,x_b+z_b)$ for a total of $2x_b+2y_b+3z_b$ terms.

Thus the number of terms in $f_{i+3}$ is equal to the sum of the number of terms of $f_i$ and $f_{i+1}$.

The fourth term in our sequence has $1+1+0=2$ terms, aligning it with $a(8)=2$, accounting for the offset.
\end{proof}

\begin{example}
We can then, as an application, compute the Kauffman bracket or Jones polynomials for the family of alternating knots where the crossings are $+-+-\ldots+-$.  In Table \ref{tab:alternating}, these are given by strings of coefficients for the appropriate terms.  It is interesting to note that for odd $b$ these are palindromic.

By work of Koseleff and Pecker \cite{PVK:fibknots}, this is the family $T(3,b,2b-3)$ with crossing number $b-1$.  They are Fibonacci knots with Schubert Fraction $F_{n}/F_{n-1}$.
\begin{table}[h]
\begin{tabular}{|c|c||c|}
\hline
$b$ & Knot &	 Coefficients of $f_b$ \\
\hline
\hline
2 &   $U$      &  (1,0) \\
4 &   $3_1$    &  (1,-1,0,-1) \\  
5 &   $4_1$    &  (1,-1,1,-1,1) \\
7 &   $6_3$    &  (-1,2,-2,3,-2,2,-1)   \\
8 &   $7_7$    &  (1,-3,3,-4,4,-3,2,-1)   \\
10&  $9_{31}$  &  (-1,4,-6,8,-10,9,-8,5,-3,1)   \\
11&  $10_{45}$ &  (-1,4,-7,11,-14,15,-14,11,-7,4,-1)   \\ 
\hline
\end{tabular}
\caption{A list of the alternating knots $T(3,b)$ with crossing number $b-1$.}
\label{tab:alternating}
\end{table}
\end{example}

\subsection{Geometric interpretation}
\label{subsec:3geom}

The Chebyshev diagram for $T(3,b,c)$ has as its associated \emph{balanced overlaid Tait graph} the $2\times (b-1)$ grid graph, as discussed in \cite{Co:clock}.  Perfect matchings on this grid graph are in one-to-one correspondence with domino tilings of a $2\times (b-1)$ board, here viewed as extending horizontally.  This is the setting that perhaps most often shows the relationship to Fibonacci numbers.

\begin{property}
\label{property:2xnFib}
The number of domino tilings of a $2\times n$ board is is $F_n$, where $F_0=1=F_1$, owing to the position of the dominoes in the final column:  a single vertical domino gives the recursion to $F(n-1)$ and two horizontal dominoes give the recurion to $F(n-2)$.
\end{property}

\begin{remark}
\label{rem:1xnFib}
In work by Benjamin, et al \cite{BEJS, BenWal}, this model is collapsed to a tiling of a $1\times n$ board with dominoes as well as $1\times 1$ squares.  The domino represents two horizontal dominoes, and the square represents a single vertical domino.
\end{remark}

Thus it should not be surprising that work in \cite{PVK:fib} relates the Alexander polynomial of 2-bridge knots to Fibonacci polynomials, as the Alexander polynomial is a determinant and perfect matchings on plane bipartite graphs correspond to terms in the determinant expansion.

We now give an alternate geometric interpretation of the model given in our Theorem \ref{thm:3rec} using a $2\times n$ board.  Recall the notation given above.
\begin{itemize}
	\item[($V$)] The $A^\pm$ is still a vertical domino.
	\item[($H$)] The $[f_2^\mp,A^\mp]$ is still two horizontal dominoes.
	\item[($C$)] The $C=[A^\pm,A^\pm]+[f_2^\mp,A^\mp]$ is thus the sum of two vertical dominoes with two horizontal dominoes.  This might be represented by a $2\times 2$ square.
	\item[($S_i$)] The two starting configurations, $f_3=[f_2^\pm,A^\pm]+[f_2^\mp,A^\mp]$, and $f_2^\pm$, might be represented by special ``capital letter'' tiles labelled by $S_2$ and $S_1$:  a $2\times 2$ square and a vertical domino, respectively.
\end{itemize}

\begin{question}
\label{question:3geom}
Can our problem be rephrased completely in terms of a combinatorial problem involving these tiles on a $2\times n$ board?
\end{question}

We can then re-interpret the previous examples:
\begin{align*}
f_4= {} & (S_2,V)+(S_1,H) \\
f_5= {} & (S_2,C)+(S_1,H,V) \\
f_6= {} & (S_2,[C,V]+[V,H])+(S_1,H,C) \\
f_7= {} & (S_2,C,C)+(S_2,V,H,V)+(S_1,H,[C,V]+[V,H]) \\
f_8= {} & (S_2,C,[C,V]+[V,H])+(S_2,V,H,C) +(S_1,H,C,C)+(S_1,H,V,H,V)\\
f_9= {} & (S_2,C,C,C)+(S_2,C,V,H,V)+(S_2,V,H,[C,V]+[V,H])\\
& +(S_1,H,C,[C,V]+[V,H]) +(S_1,H,V,H,C)\\
f_{10}= {} & (S_2,C,C,[C,V]+[V,H])+(S_2,C,V,H,C) +(S_2,V,H,C,C)+(S_2,V,H,V,H,V)\\
&  +(S_1,H,C,C,C)+(S_1,H,C,V,H,V)+(S_1,H,V,H,[C,V]+[V,H]) 
\end{align*}

\section{The height-4 case}
\label{sec:4}

We do not need to consider Chebyshev knots $T(4,b)$ except for a base case, the link $T(4,2,\pm\pm)$, that will appear in the following section.

What is interesting here is that from the billiard table diagram perspective we have two components of the link that are long knots, unlike the $a=3$ case where only one component is ever a long knot.  This requires us to make some decisions about the closures of our tangles.

A $k$-\emph{tangle} is an embedding of some $k$ arcs and any number of circles into $\mathbb{R}\times[0,1]$.  The knots studied in the last section, $2$-bridge knots, are examples of $2$-tangles, and these are often closed into knots using the \emph{numerator and denominator closures} which have no crossings.  The numerator closure identifies the two ``north'' endpoints and the two ``south'' endpoints, and the denominator closure identifies the two ``east'' endpoints and the two ``west'' endpoints.

However, the motivation from long knots is that the closure of two arcs happens at infinity.  Thus in our case with two components that are long knots, if the closures involve some crossing, one could interpret this as a single virtual crossing at infinity.

Perhaps it is this ambiguity that prompted Koseleff and Pecker to consider solely the case of knots, when $a$ and $b$ are co-prime, and when one need not add extra components to complete the billiard table diagram.

In our discussion, however, several $2$-tangles appear.  We will use closures that do not involve any additional crossings and attempt to show that this setting is correct.

Specifically for the following example, we insist on the convention of identifying the two ends of each of the two arcs.

\begin{example}
\label{ex:T42}
Consider the link $T(4,2,\pm\pm)$, where the two crossings are independent of each other, with the denominator closure of the tangle.  The Kauffman bracket polynomial is
\begin{equation}
\label{eq:42}
g_2=\delta(A^\pm,A^\pm)+(A^\pm,A^\mp)+(A^\mp,A^\pm)+\delta(A^\mp,A^\mp)=([X]+\delta[A^\mp,A^\mp]),
\end{equation}
where
\begin{equation}
\label{eq:X}
X=\delta[A^\pm,A^\pm]+[A^\pm,A^\mp]+[A^\mp,A^\pm]=\begin{cases}
1-(A^{\pm})^4 & \text{for $++$ or $--$, respectively,} \\
0 &\text{for $+-$ or $-+$.}
\end{cases}
\end{equation}

When the two signs agree we get $g_2=\delta(-\delta)+(2)=2-\delta^2=-A^4-A^{-4}$ for the Hopf link.  When they don't we get $g_2=\delta(2)+(-\delta)=\delta=-A^2-A^{-2}$ for the unlink.
\end{example}

\section{The height-5 case and bumpered billiard tables}
\label{sec:5}

In this section we consider Chebyshev knots $T(5,b)$ and their Kauffman bracket polynomials, denoted $h_b$.  The crossings are labelled from bottom to top in each column and then from left to right on the columns.  However, a new object will appear in the recursion, and we begin by defining these.

\begin{definition}
\label{def:pocket}
Let $B^n(a,b)$ and $B_n(a,b)$ be $(a-1)\times (b-1)$ billiard tables with $n$ ``bumpers'' or squares removed from the last column, starting at the top or bottom, respectively.  We call these \emph{bumpered billiard tables}.  Billiard balls can only leave the table at corner ``pockets,'' as before, although now there are some new corners.
\end{definition}

\begin{remark}
\label{rem:corners}
The examples we consider specifically avoid the situation where a crossing occurs at a new interior corner ``pocket'' so that there is no ambiguity as to what occurs there.
\end{remark}

\begin{example}
\label{ex:BUp257}
Consider the bumpered billiard table for $B^2(5,7)$, with two squares removed from the top of the last column, as shown in Figure \ref{fig:B2}A.  Note that there is still a single component here; this is not always the case for general bumpered billiard tables.

Consider the bumpered billiard table for $B_2(5,8)$, with two squares removed from the bottom of the last column, as shown in Figure \ref{fig:B2}B.  Note that there are two components here but that only one of them is a long knot.

We consider the class $B2(5,b)$, where we have $B^2(5,b)$ for odd $b$ and $B_2(5,b)$ for even $b$.  Observe that in order to satisfy Remark \ref{rem:corners}, the removal of two squares must come from the top for odd $b$ and from the bottom for even $b$.

The Kauffman bracket polynomials for this class are treated generally in Lemma \ref{prop:B25b} below.
\end{example}

\renewcommand{\thesubfigure}{}
\begin{figure}[htbp]
\centering
\subfigure[A.  Bumpered billiard table $B^2(5,7)$.]{
\begin{tikzpicture}[scale=.6]
    \foreach \i in {0,...,7} {
        \draw [very thin,lightgray] (\i,0) -- (\i,5);
    }
    \foreach \i in {0,...,5} {
        \draw [very thin,lightgray] (0,\i) -- (7,\i);
    }
\fill[white] (6.1,3.1) rectangle (7,5);
\draw[color=white](7.23,5.25) node[above right] {$B$};
\draw[thick](-.25,-.25) -- (5,5) -- (6,4) -- (2,0) -- (0,2) -- (3,5) -- (7,1) -- (6,0) -- (1,5) -- (0,4) -- (4,0) -- (7,3) -- (7.25,3.25);
\end{tikzpicture}
}
\hspace{2em}
\subfigure[B.  Bumpered billiard table $B_2(5,8)$ with two components.]{
\begin{tikzpicture}[scale=.6]
    \foreach \i in {0,...,8} {
        \draw [very thin,lightgray] (\i,0) -- (\i,5);
    }
    \foreach \i in {0,...,5} {
        \draw [very thin,lightgray] (0,\i) -- (8,\i);
    }
\fill[white] (7.1,0) rectangle (8,1.9);
\draw[thick, dashed] (7,1) -- (3,5) -- (0,2) -- (2,0) -- (7,5) -- (8,4) -- (4,0) -- (0,4) -- (1,5) -- (6,0) -- cycle;
\draw[thick] (-.25,-.25) -- (5,5) -- (8.25,1.75);
\end{tikzpicture}
}
\caption{\label{fig:B2} Bumpered billiard tables with two squares removed.}
\end{figure}
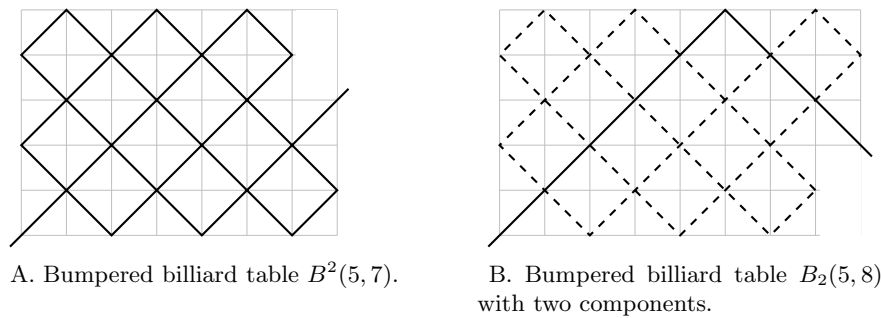

\renewcommand{\thesubfigure}{}
\begin{figure}[htbp]
\centering
\subfigure[A.  Bumpered billiard table $B_1(5,7)$ with two components, forming a 2-tangle.]{
\begin{tikzpicture}[scale=.6]
    \foreach \i in {0,...,7} {
        \draw [very thin,lightgray] (\i,0) -- (\i,5);
    }
    \foreach \i in {0,...,5} {
        \draw [very thin,lightgray] (0,\i) -- (7,\i);
    }
\fill[white] (6.1,0) rectangle (7,.9);
\draw[thick] (-.25,-.25) node {\textbullet} -- (5,5) -- (7,3) -- (4,0) -- (0,4) -- (1,5) -- (6,0);
\draw[thick, dashed] (7,1) -- (3,5) -- (0,2) -- (2,0) -- (7.23,5.25) node {$*$};
\draw[thick, dashed](7,1) -- (7.25,0.75) node {$*$};
\draw[thick](6,0) -- (6.25,-0.25) node {\textbullet};
\end{tikzpicture}
}
\hspace{2em}
\subfigure[B.  Bumpered billiard table $B^1(5,8)$ with two components, forming a 2-tangle.]{
\begin{tikzpicture}[scale=.6]
    \foreach \i in {0,...,8} {
        \draw [very thin,lightgray] (\i,0) -- (\i,5);
    }
    \foreach \i in {0,...,5} {
        \draw [very thin,lightgray] (0,\i) -- (8,\i);
    }
\fill[white] (7.1,4.04) rectangle (8,5);
\draw[thick](-.25,-.25) node {\textbullet} -- (5,5) -- (8,2) -- (6,0) -- (1,5) -- (0,4) -- (4,0) -- (8.25,4.25) node {$*$};
\draw[thick, dashed] (8.25,-.25) node {$*$} -- (3,5) -- (0,2) -- (2,0) -- (7.23,5.25) node {\textbullet};
\draw[thick](7,3) -- (7.25,3.25);
\end{tikzpicture}
}
\caption{\label{fig:B1} Bumpered billiard tables with one square removed, forming 2-tangles.  We choose as closures those that are preserved in the classical rectangular billiard tables, here identified by similar symbols.}
\end{figure}
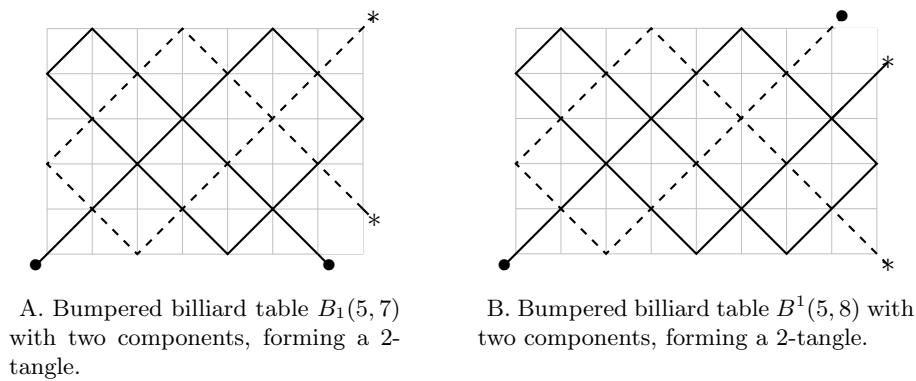

\begin{example}
\label{ex:BDown157}
Consider the bumpered billiard table for $B_1(5,7)$, with one square removed from the bottom of the last column, as shown in Figure \ref{fig:B1}A.  Note that there are now two components, yielding a 2-tangle.  We will choose to identify the ends of these tangles following the corresponding symbols in the figure, as this closure is preserved in the classical rectangular billiard tables.

Consider the bumpered billiard table for $B^1(5,8)$, with one square removed from the top of the last column, as shown in Figure \ref{fig:B1}B.  Note that there are now two components, yielding a 2-tangle.  We will choose to identify the ends of these tangles following the corresponding symbols in the figure, as this closure is preserved in the classical rectangular billiard tables.

We consider the class $B1(5,b)$, where we have $B^1(5,b)$ for even $b$ and $B_1(5,b)$ for odd $b$.  Observe that in order to satisfy Remark \ref{rem:corners}, the removal of a single square must come from the top for even $b$ and from the bottom for odd $b$.

The Kauffman bracket polynomials for this class are treated generally in Lemma \ref{prop:B15b} below.
\end{example}

\subsection{Small examples}
\label{subsec:5ex}
We consider the examples $T(5,b)$, $B2(5,b)$, and $B1(5,b)$ with small values of $b$.  Along the way, we introduce some notation and ideas that will be useful below.

\begin{example}
\label{ex:T51-2}
The knot $T(5,1,\emptyset)$ is the unknot and thus has Kauffman bracket and Jones polynonials equal to 1.

All four knots $T(5,2,\pm\pm)$ are the unknot with two Reidemeister I moves and thus have Jones polynonial equal to 1.  The Kauffman bracket polynomial is
\begin{equation}
\label{eq:T52}
h_2=(A^\pm,A^\pm)+\delta(A^\pm,A^\mp)+\delta(A^\mp,A^\pm)+\delta^2(A^\mp,A^\mp).
\end{equation}

Note that when the crossings have different signs, we get $h_2=1$.  When they are both positive or both negative, we get $h_2=A^{\mp6}$, respectively.
\end{example}

\begin{example}
\label{ex:BUp253}
Consider the bumpered billiard link $B^2(5,3,\pm\pm\pm)$, where the three crossings are independent of each other.  The Kauffman bracket polynomial is
\begin{equation}
\label{eq:BUp253}
b_3=(h_2,A^\pm)+(f_2^\mp,f_2^\pm,A^\mp).
\end{equation}
Let $M=[f_2^\mp,f_2^\pm,A^\mp]$ so that $b_3=(h_2,A^\pm)+(M)$.

Consider the bumpered billiard knot $B_1(5,3,\pm\pm\pm\pm)$, where the four crossings are independent of each other.  We choose the tangle closure as in Figure \ref{fig:B1}A.  The Kauffman bracket polynomial is
\begin{equation}
\label{eq:Down153}
bt_3=(h_2,X)+(f_2^\pm,f_2^\mp,A^\mp,A^\mp).
\end{equation}
Let $S=[f_2^\pm,f_2^\mp,A^\mp,A^\mp]$ so that $bt_3=(h_2,X)+(S)$.
\end{example}

\begin{example}
\label{ex:T53}
Consider the knot $T(5,3,\pm\pm\pm\pm)$, where the four crossings are independent of each other.  The Kauffman bracket polynomial is
\begin{equation}
\label{eq:53}
h_3=(h_2,A^\pm,A^\pm)+(f_2^\mp,f_2^\mp,A^\mp,A^\mp)+(g_2,A^\pm,A^\mp)+(f_2^\mp,f_2^\pm,A^\mp,A^\pm).
\end{equation}

Recall that $g_2=[X]+\delta[A^\mp,A^\mp]$ is the Kauffman bracket polynomial of $T(4,2)$ from Equation \ref{eq:42} of Section \ref{sec:4}.  Let $K=[f_2^\mp,f_2^\mp,A^\mp,A^\mp]$ so that $h_3=(h_2,A^\pm,A^\pm)+(K)+(g_2,A^\pm,A^\mp)+(M,A^\pm)$.
\end{example}

\begin{example}
\label{ex:BUp254}
Consider the bumpered billiard knot $B^2(5,4,\pm\pm\pm\pm\pm)$, where the five crossings are independent of each other.  The Kauffman bracket polynomial is
\begin{equation}
\label{eq:BUp254}
b_4=(h_3,\_,A^\pm)+(h_2,A^\pm,f_2^\mp,\_,A^\mp)+(M,f_2^\mp,\_,A^\mp),
\end{equation}
where here a $\_$ in the $i$-th position denotes that there is no $i$-th crossing following the standard ordering, due to the removed bumpers.

This notation also appears below when some giving information for later crossings after some $B2(5,b)$.

Consider the bumpered billiard knot $B_1(5,4,\pm\pm\pm\pm\pm\pm)$, where the six crossings are independent of each other.  We choose the tangle closure as in Figure \ref{fig:B1}B.  The Kauffman bracket polynomial is
\begin{equation}
\label{eq:Down154}
bt_4=(h_3,X)+(h_2,f_2^\mp,A^\pm,A^\mp,A^\mp)+(g_2,f_2^\mp,A^\mp,A^\mp,A^\mp).
\end{equation}
Let $R=[f_2^\mp,A^\pm,A^\mp,A^\mp]$ and $N=[f_2^\mp,A^\mp,A^\mp,A^\mp]$ so that $bt_4=(h_3,X)+(h_2,R)+(g_2,N)$.
\end{example}

\begin{example}
\label{ex:T54}
Consider the knot $T(5,4,\pm\pm\pm\pm\pm\pm)$, where the six crossings are independent of each other.  The Kauffman bracket polynomial is
\begin{equation}
\label{eq:54}
h_4=(h_3,A^\pm,A^\pm)+(h_2,K)+(b_3,f_2^\mp,A^\pm,A^\mp)+(bt_3,A^\mp,A^\pm).
\end{equation}
Let $L=[f_2^\mp,A^\pm,A^\mp]$ so that $h_4=(h_3,A^\pm,A^\pm)+(h_2,K)+(b_3,L)+(bt_3,A^\mp,A^\pm)$.  Then
\begin{align*}
h_4 = {} & (h_3,A^\pm,A^\pm)+(h_2,K)+(h_2,A^\pm,L)+(M,L)+(h_2,X,A^\mp,A^\pm)+(S,A^\mp,A^\pm)\\
= {} & (h_3,A^\pm,A^\pm)+(h_2,[K]+[A^\pm,L]+[X,A^\mp,A^\pm])+(M,L)+(S,A^\mp,A^\pm).
\end{align*}
Let $P'_1=\widetilde{P}'_1=[A^\pm,A^\pm]$ and $P'_2=[K]+[A^\pm,L]+[X,A^\mp,A^\pm]$ so that $h_4=(h_3,P'_1)+(h_2,P'_2)+(M,L)+(S,A^\mp,A^\pm)$.
\end{example}

\begin{example}
\label{ex:BUp255}
Consider the bumpered billiard link $B^2(5,5,\pm\pm\pm\pm\pm\pm\pm)$, where the seven crossings are independent of each other.  The Kauffman bracket polynomial is
\begin{equation}
\label{eq:BUp255}
b_5=(h_4,A^\pm)+(h_3,L)+(h_2,A^\pm,K)+(M,K).
\end{equation}

Consider the bumpered billiard link $B_1(5,5,\pm\pm\pm\pm\pm\pm\pm\pm)$, where the eight crossings are independent of each other.  We choose the tangle closure as in Figure \ref{fig:B1}A.  The Kauffman bracket polynomial is
\begin{equation}
\label{eq:Down155}
bt_5=(h_4,X)+(h_3,A^\pm,f_2^\mp,A^\mp,A^\mp)+(h_2,X,A^\mp,f_2^\mp,A^\mp,A^\mp)+(S,A^\mp,f_2^\mp,A^\mp,A^\mp).
\end{equation}
Let $\widetilde{R}=[A^\pm,f_2^\mp,A^\mp,A^\mp]$ and $\widetilde{N}=[A^\mp,f_2^\mp,A^\mp,A^\mp]$ so that $bt_5=(h_4,X)+(h_3,\widetilde{R})+(h_2,X,\widetilde{N})+(S,\widetilde{N})$.
\end{example}

\begin{example}
\label{ex:T55}
Consider the link $T(5,5,\pm\pm\pm\pm\pm\pm\pm\pm)$, where the eight crossings are independent of each other.  The Kauffman bracket polynomial is
\begin{equation}
\label{eq:55}
h_5=(h_4,P'_1)+(h_3,K)+(bt_4,A^\pm,A^\mp)+(b_4,f_2^\mp,\_,A^\mp,A^\pm),
\end{equation}
\begin{align*}
= {} & (h_3,P'_1,P'_1)+(h_2,P'_2,P'_1)+(M,L,P'_1)+(S,A^\mp,A^\pm,P'_1)+(h_3,K)\\
& +(h_3,X,A^\pm,A^\mp)+(h_2,R,A^\pm,A^\mp)+(g_2,K,A^\pm,A^\mp) \\
& +(h_3,L,A^\pm) +(h_2,A^\pm,K,A^\pm)+(M,K,A^\pm) \\
= {} & (h_3,P'_1,P'_1)+(h_3,[X,A^\pm,A^\mp]+[L,A^\pm]+[K])\\
& +(h_2,P'_2,P'_1)+(h_2,[R,A^\pm,A^\mp]+[A^\pm,K,A^\pm])\\
& +(M,L,P'_1)+(M,K,A^\pm)+(S,A^\mp,A^\pm,P'_1)+(g_2,K,A^\pm,A^\mp). \\
\end{align*}
Let $\widetilde{P}'_2=[X,A^\pm,A^\mp]+[L,A^\pm]+[K]$ and $P'_3=[R,A^\pm,A^\mp]+[A^\pm,K,A^\pm]$ so that $h_5=(h_3,P'_1,P'_1)+(h_3,\widetilde{P}'_2)+(h_2,P'_2,P'_1)+(h_2,P'_3)+(M,L,P'_1)+(M,K,A^\pm)+(S,A^\mp,A^\pm,P'_1)+(g_2,K,A^\pm,A^\mp)$.
\end{example}

As stated in the introduction, these blocks $P_i$ and $\widetilde{P}_i$ will appear as elements of size $i$ in a partition $\mathcal{P}_n$ of $n$ blocks.  See the next section.  To summarize, they are so far:
\begin{align*}
P'_1= {} & \widetilde{P}'_1 = [A^\pm,A^\pm]\\
P'_2= {} & [K]+[A^\pm,L]+[X,A^\mp,A^\pm]\\
= {} & [f_2^\mp,f_2^\mp,A^\mp,A^\mp]+[A^\pm,f_2^\mp,A^\pm,A^\mp]+\delta[A^\pm,A^\pm,A^\mp,A^\pm]+[A^\pm,A^\mp,A^\mp,A^\pm]+[A^\mp,A^\pm,A^\mp,A^\pm]\\
\widetilde{P}'_2= {} & [X,A^\pm,A^\mp]+[L,A^\pm]+[K]\\
= {} & \delta[A^\pm,A^\pm,A^\pm,A^\mp]+[A^\pm,A^\mp,A^\pm,A^\mp]+[A^\mp,A^\pm,A^\pm,A^\mp]+[f_2^\mp,A^\pm,A^\mp,A^\pm]+[f_2^\mp,f_2^\mp,A^\mp,A^\mp]\\
P'_3= {} & [R,A^\pm,A^\mp]+[A^\pm,K,A^\pm] \\
= {} & [f_2^\mp,A^\pm,A^\mp,A^\mp,A^\pm,A^\mp]+[A^\pm,f_2^\mp,f_2^\mp,A^\mp,A^\mp,A^\pm]
\end{align*}

\section{Main results on 3-bridge knots}
\label{subsec:results}

At last we have the machinery to arrive at the Main Theorem \ref{thm:main}, Notation \ref{note:thm}, and Proposition \ref{prop:exponential}.  The proof of the main theorem appears in Subsection \ref{subsec:proofs}, after two useful lemmas giving recursions for the Kauffman bracket polynomials of bumpered billiard tables in Subsection \ref{subsec:lemmasbumpered}.

\begin{theorem}
\label{thm:main}
(Main Theorem) Following Notation \ref{note:thm} below, the Kauffman bracket polynomials of $T(5,b)$ (for indeterminate crossing information), including the class of 3-bridge knots, are
\begin{align*}
h_b= {} & \sum_{i=3}^{b-1}([h_3,P_{i-2}]+[h_2,P_{i-1}]+[Q_i],\mathcal{P}_{b-1-i})
\end{align*}
for $b\geq 4$ with base cases $h_1=1$, $h_2=(A^\pm,A^\pm)+\delta(A^\pm,A^\mp)+\delta(A^\mp,A^\pm)+\delta^2(A^\mp,A^\mp)$,  and $h_3=(h_2,A^\pm,A^\pm)+(f_2^\mp,f_2^\mp,A^\mp,A^\mp)+(g_2,A^\pm,A^\mp)+(f_2^\mp,f_2^\pm,A^\mp,A^\pm)$, where $\delta=(-A^2-A^{-2})$.

In particular, here is a small list of Kauffman brackets for this class:
\begin{align*}
h_4= {} & (h_3,P_1)+(h_2,P_2)+(M,L)+(S,A^\mp,A^\pm) \\
h_5= {} & ([h_3,P_1]+[h_2,P_2]+[M,L]+[S,A^\mp,A^\pm],P_1)\\
{} & +(h_3,\widetilde{P}_2)+(h_2,P_3)+(M,K,A^\pm)+(g_2,K,A^\pm,A^\mp)\\
h_6= {} & ([h_3,P_1]+[h_2,P_2]+[M,L]+[S,A^\mp,A^\pm],[P_2]+[P_1,P_1]) \\
{} & +([h_3,\widetilde{P}_2]+[h_2,P_3]+[M,K,A^\pm]+[g_2,K,A^\pm,A^\mp],P_1)\\
{} & +(h_3,\widetilde{P}_3)+(h_2,P_4)+(M,K,L) +(S,\widetilde{N},A^\mp,A^\pm)
\end{align*}
\end{theorem}

\begin{notation}
\label{note:thm}
Let $\square^i$ be $[\square,\square,\ldots,\square]$ written $i$ times.  Let $f_2^\pm=-A^{\mp3}$ and $g_2$ be the Kauffman bracket polynomials of $T(3,2)$ and $T(4,2)$, respectively.  Let $X=\delta[A^\pm,A^\pm]+[A^\pm,A^\mp]+[A^\mp,A^\pm]$.

Let $P'_1=\widetilde{P}'_1=[A^\pm,A^\pm]$, and the remaining $P'_i$ and $\widetilde{P}'_i$ be as follows: \\
$P'_2=[f_2^\mp,f_2^\mp,A^\mp,A^\mp]+[A^\pm,f_2^\mp,A^\pm,A^\mp]+\delta[A^\pm,A^\pm,A^\mp,A^\pm]+[A^\pm,A^\mp,A^\mp,A^\pm]+[A^\mp,A^\pm,A^\mp,A^\pm]$,\\
$\widetilde{P}'_2=[f_2^\mp,f_2^\mp,A^\mp,A^\mp] +[f_2^\mp,A^\pm,A^\mp,A^\pm] +\delta[A^\pm,A^\pm,A^\pm,A^\mp]+[A^\mp,A^\pm,A^\pm,A^\mp]+[A^\pm,A^\mp,A^\pm,A^\mp]$, 

\begin{equation*}
P'_i=
\begin{cases}
[f_2^\mp,A^\pm,A^\mp,A^\mp,[f_2^\mp,A^\mp,A^\mp,A^\mp]^{j},A^\pm,A^\mp] & \\
\mbox{\hspace{5mm}} + [A^\pm,[f_2^\mp,f_2^\mp,A^\mp,A^\mp]^{j+1},A^\pm] &\mbox{ for $i=2j+3$ odd and }\\
[A^\pm,[f_2^\mp,f_2^\mp,A^\mp,A^\mp]^{j},f_2^\mp,A^\pm,A^\mp] & \\
\mbox{\hspace{5mm}} + [X,[f_2^\mp,A^\mp,A^\mp,A^\mp]^{j},A^\mp,A^\pm] &\mbox{ for $i=2j+2$ even. }
\end{cases}
\end{equation*}

\begin{equation*}
\widetilde{P}'_i=
\begin{cases}
[f_2^\mp,A^\pm,A^\mp,[f_2^\mp,f_2^\mp,A^\mp,A^\mp]^{j},f_2^\mp,A^\pm,A^\mp] & \\
\mbox{\hspace{5mm}} + [A^\pm,f_2^\mp,A^\mp,A^\mp,[A^\mp,f_2^\mp,A^\mp,A^\mp]^{j},A^\mp,A^\pm] &\mbox{ for $i=2j+3$ odd and }\\
[f_2^\mp,A^\pm,A^\mp,[f_2^\mp,f_2^\mp,A^\mp,A^\mp]^{j},A^\pm] & \\
\mbox{\hspace{5mm}} + [X,[f_2^\mp,A^\mp,A^\mp,A^\mp]^{j},A^\pm,A^\mp] &\mbox{ for $i=2j+2$ even. }
\end{cases}
\end{equation*}

The blocks $P'_i$ and $\widetilde{P}'_i$ correspond to elements of size $i$ in $\mathcal{P}_n$, the set of all partitions of the integer $n$.  
  Let $j$ be the position in $n$ of the first part of one of these elements.  For convenience of notation we set 
\begin{equation*}
P_i:=P_i(j)=
\begin{cases}
P'_i  &\mbox{ if $i+j$ is odd and }\\
\widetilde{P}'_i  &\mbox{ if $i+j$ is even. }
\end{cases}
\end{equation*}

For $i\geq 3$, let $Q_i$ be as follows:
\begin{equation*}
Q_i=
\begin{cases}
[f_2^\mp,f_2^\pm,A^\mp,[f_2^\mp,f_2^\mp,A^\mp,A^\mp]^{j},f_2^\mp,A^\pm,A^\mp] &\\
\mbox{\hspace{5mm}} +[f_2^\pm,f_2^\mp,A^\mp,A^\mp,[A^\mp,f_2^\mp,A^\mp,A^\mp]^{j},A^\mp,A^\pm] &\mbox{ for $i=2j+3$ odd and }\\
[f_2^\mp,f_2^\pm,A^\mp,[f_2^\mp,f_2^\mp,A^\mp,A^\mp]^{j},A^\pm] & \\
\mbox{\hspace{5mm}} + [g_2,[f_2^\mp,A^\mp,A^\mp,A^\mp]^{j},A^\pm,A^\mp] &\mbox{ for $i=2j$ even. }
\end{cases}
\end{equation*}

For completion, here are some other variables that appear above:
\begin{align*}
K=[f_2^\mp,f_2^\mp,A^\mp,A^\mp],  {\mbox{\hspace{10mm}}} & N=[f_2^\mp,A^\mp,A^\mp,A^\mp], \\
L=[f_2^\mp,A^\pm,A^\mp], {\mbox{\hspace{10mm}}} &  \widetilde{N}=[A^\mp,f_2^\mp,A^\mp,A^\mp], \\
M=[f_2^\mp,f_2^\pm,A^\mp], {\mbox{\hspace{10mm}}} &  R=[f_2^\mp,A^\pm,A^\mp,A^\mp], \\
S=[f_2^\pm,f_2^\mp,A^\mp,A^\mp],  {\mbox{\hspace{10mm}}} &  \widetilde{R}=[A^\pm,f_2^\mp,A^\mp,A^\mp].
\end{align*}
Note that this allows one to write:
\begin{equation*}
P'_i=
\begin{cases}
[R,N^j,A^\pm,A^\mp]+[A^\pm,K^{j+1},A^\pm] &\mbox{ for $i=2j+3$ odd and }\\
[X,N^j,A^\mp,A^\pm]+[A^\pm,K^j,L] &\mbox{ for $i=2j+2$ even. }
\end{cases}
\end{equation*}
\begin{equation*}
\widetilde{P}'_i=
\begin{cases}
[L,K^j,L]+[\widetilde{R},\widetilde{N}^j,A^\mp,A^\pm] &\mbox{ for $i=2j+3$ odd and }\\
[L,K^j,A^\pm]+[X,N^j,A^\pm,A^\mp] &\mbox{ for $i=2j+2$ even. }
\end{cases}
\end{equation*}
\begin{equation*}
Q_i=
\begin{cases}
[M,K^{j},L]+[S,\widetilde{N}^{j},A^\mp,A^\pm] &\mbox{ for $i=2j+3$ odd and }\\
[M,K^{j},A^\pm]+[g_2,N^j,A^\pm,A^\mp] &\mbox{ for $i=2j$ even. }
\end{cases}
\end{equation*}
\end{notation}

\begin{proposition}
\label{prop:exponential}
The number of terms in the expansion of the Kauffman bracket polynomial $h_b$ of $T(5,b)$ in Theorem \ref{thm:main} is $2^{b-4}$.
\end{proposition}

\begin{proof}
Observe that $|\mathcal{P}_j|=2^{j-1}$.  Then the sum over all $\mathcal{P}_j$ gives $1+\sum_{j=1}^{b-4} 2^{j-1}=2^{b-4}$ terms.
\end{proof}

\begin{remark}
\label{rem:terms3}
Note that this is far fewer than the usual $2^{2(b-1)}$ terms using the Skein relation for the $2(b-1)$ crossings of the diagram.
\end{remark}

\begin{example}
We can then, as an application, compute the Kauffman bracket or Jones polynomials for the family of alternating Chebyshev knots where the crossings are $++--++--\ldots++--$.  In Table \ref{tab:alternating2}, these are given by strings of coefficients for the appropriate terms.
\begin{table}[h]
\begin{tabular}{|c|c||c|}
\hline
$b$ & Knot &	 Coefficients of $h_b$ \\
\hline
\hline
2 &   $U$      &  (1) \\
3 &   $4_1$    &  (1,-1,1,-1,1) \\
4 &   $6_2$    &  (1,-1,2,-2,2-2,1) \\  
6 & $10_{116}$ &  (1,-4,8,-11,15,-16,15,-12,8,-4,1)   \\
7 & $12a_{0960}$ &  (1,-5,13,-23,34,-42,45,-42,34,-23,13,-5,1)   \\
\hline
\end{tabular}
\caption{A list of the alternating knots $T(5,b)$ with crossing number $2(b-1)$.}
\label{tab:alternating2}
\end{table}
\end{example}

\begin{question}
\label{question:5geom}
Can our problem be rephrased in terms of a combinatorial problem involving these tiles on a $4\times n$ board?
\end{question}

\begin{question}
\label{ques:unknot}
Can this model be used to show that the Jones polynomial detects the unknot for knots that ($k\leq 3$)-bridge?
\end{question}

\subsection{Lemmas on bumpered billiard tables}
\label{subsec:lemmasbumpered}

The following two lemmas on the Kauffman bracket polynomials of bumpered billiard tables give recursions in terms of the Kauffman bracket polynomials of rectangular billiard tables and are used for the main proof in the next subsection.

For convenience in this subsection, we replace the variable $b$ with the variable $n$ so that the Kauffman bracket polynomials are $b_n$ and $bt_n$ for the two classes.

We first consider the class $B2(5,n)$, where we have $B^2(5,n)$ for odd $n$ and $B_2(5,n)$ for even $n$.  Although this may appear strange, it follows naturally from the recursion.

\begin{lemma}
\label{prop:B25b}
The Kauffman bracket polynomials of $B2(5,n)$ (for indeterminate crossing information) are:
\begin{align}
b_{n\text{ odd}}= {} &  (M,K^{\frac{n-3}{2}})+\sum_{i=0}^\frac{n-3}{2}(h_{n-1-2i},A^\pm,K^i) +\sum_{i=0}^\frac{n-5}{2}(h_{n-2-2i},L,K^i), \\
b_{n\text{ even}}= {} & (h_{n-1},\_,A^\pm) +(M,K^{\frac{n-4}{2}},f_2^\mp,\_,A^\mp)\\
& +\sum_{i=0}^\frac{n-4}{2}(h_{n-2-2i},A^\pm,K^i,f_2^\mp,\_,A^\mp) +\sum_{i=0}^\frac{n-6}{2}(h_{n-3-2i},L,K^i,f_2^\mp,\_,A^\mp),
\end{align}
for $n\geq5$ and with base cases $b_1=1$, $b_2=(\_,f_2^\pm)$, $b_3=(h_2,A^\pm)+(M)$, and $b_4=(h_3,\_,A^\pm)+(h_2,A^\pm,f_2^\mp,\_,A^\mp)+(M,f_2^\mp,\_,A^\mp)$, where $\square^i$ represents $[\square,\square,\ldots,\square]$, written $i$ times, and where $K=[f_2^\mp,f_2^\mp,A^\mp,A^\mp]$, $L=[f_2^\mp,A^\pm,A^\mp]$, and $M=[f_2^\mp,f_2^\pm,A^\mp]$.

In particular, here is a list of Kauffman brackets for this class:
\begin{align*}
b_3= {} & (h_2,A^\pm)+(M) \\
b_4= {} & (h_3,\_,A^\pm)+(h_2,A^\pm,f_2^\mp,\_,A^\mp) +(M,f_2^\mp,\_,A^\mp) \\
b_5= {} & (h_4,A^\pm)+(h_3,L)+(h_2,A^\pm,K) +(M,K) \\
b_6= {} & (h_5,\_,A^\pm)+(h_4,A^\pm,f_2^\mp,\_,A^\mp)+(h_3,L,f_2^\mp,\_,A^\mp)+(h_2,A^\pm,K,f_2^\mp,\_,A^\mp) +(M,K,f_2^\mp,\_,A^\mp) \\
b_7= {} & (h_6,A^\pm)+(h_5,L)+(h_4,A^\pm,K)+(h_3,L,K) +(h_2,A^\pm,K,K) +(M,K,K) \\
b_8= {} & (h_7,\_,A^\pm)+(h_6,A^\pm,f_2^\mp,\_,A^\mp)+(h_5,L,f_2^\mp,\_,A^\mp)+(h_4,A^\pm,K,f_2^\mp,\_,A^\mp)\\
& +(h_3,L,K,f_2^\mp,\_,A^\mp) +(h_2,A^\pm,K,K,f_2^\mp,\_,A^\mp) +(M,K,K,f_2^\mp,\_,A^\mp) \\
b_9= {} & (h_8,A^\pm)+(h_7,L)+(h_6,A^\pm,K)+(h_5,L,K)+(h_4,A^\pm,K,K)+(h_3,L,K,K) \\
&  +(h_2,A^\pm,K,K,K)+(M,K,K,K)\\
b_{10}= {} & (h_9,\_,A^\pm)+(h_8,A^\pm,f_2^\mp,\_,A^\mp)+(h_7,L,f_2^\mp,\_,A^\mp)+(h_6,A^\pm,K,f_2^\mp,\_,A^\mp) \\
& +(h_5,L,K,f_2^\mp,\_,A^\mp)+(h_4,A^\pm,K,K,f_2^\mp,\_,A^\mp)+(h_3,L,K,K,f_2^\mp,\_,A^\mp)  \\
& +(h_2,A^\pm,K,K,K,f_2^\mp,\_,A^\mp)+(M,K,K,K,f_2^\mp,\_,A^\mp)
\end{align*}
\end{lemma}

\begin{proof}
The case $B^2(5,1)$ is equivalent to $T(3,1)$, and so $b_1=1$.  The case $B_2(5,2)$ is equivalent to $T(3,2)$ after a ``skipped'' first crossing, and so we get $b_2=(\_,f_2^\pm)$.  The actual base cases $B^2(5,3)$ and $B_2(5,4)$ for the induction are handled in Examples \ref{ex:BUp253} and \ref{ex:BUp254}.

Prior to addressing the induction hypothesis, consider first the two resolutions on the last crossing.  One of them results in a rectangular billiard table of width one less, and the other results in a nugatory crossing that, when resolved, achieves a \emph{bumpered} billiard table of width one less.  This is summarized in the following equations:
\begin{align}
b_{n\text{ odd}}= {} & (h_{n-1},A^\pm)+(b_{n-1},f_2^\mp,\_,A^\mp) \text{ and} \\
b_{n\text{ even}}= {} & (h_{n-1},\_,A^\pm)+(b_{n-1},f_2^\mp,\_,A^\mp),
\end{align}
where for odd $n$, the $f_2^\mp$ is ordered \emph{before the last term} of the $b_{n-1}$ due to placement of the last crossing within the billiard table.

Suppose the induction hypothesis holds for even $n$.  Then for odd $n+1$,
\begin{align*}
b_{n+1\text{ odd}}= {} & (h_{n},A^\pm)+(b_{n},f_2^\mp,\_,A^\mp) \\
= {} & (h_{n},A^\pm) + (h_{n-1},L) +(M,K^{\frac{n-4}{2}+1})\\
& +\sum_{i=0}^\frac{n-4}{2}(h_{n-2-2i},A^\pm,K^{i+1}) +\sum_{i=0}^\frac{n-6}{2}(h_{n-3-2i},L,K^{i+1}) \\
= {} & (h_{n},A^\pm) + (h_{n-1},L) +(M,K^{\frac{n-4}{2}+1})\\
& +\sum_{i=1}^\frac{n-2}{2}(h_{n-2i},A^\pm,K^{i}) +\sum_{i=1}^\frac{n-4}{2}(h_{n-1-2i},L,K^{i}) \\
= {} & +\sum_{i=0}^\frac{n-2}{2}(h_{n-2i},A^\pm,K^{i}) +\sum_{i=0}^\frac{n-4}{2}(h_{n-1-2i},L,K^{i}) +(M,K^{\frac{n-4}{2}+1}),
\end{align*}
as desired.

Suppose the induction hypothesis holds for odd $n$.  Then for even $n+1$,
\begin{align*}
b_{n+1\text{ even}}= {} & (h_{n},\_,A^\pm)+(b_{n},f_2^\mp,\_,A^\mp) \\
= {} & (h_{n},\_,A^\pm)+(M,K^{\frac{n-3}{2}},f_2^\mp,\_,A^\mp)\\
& +\sum_{i=0}^\frac{n-3}{2}(h_{n-1-2i},A^\pm,K^i,f_2^\mp,\_,A^\mp) +\sum_{i=0}^\frac{n-5}{2}(h_{n-2-2i},L,K^i,f_2^\mp,\_,A^\mp),
\end{align*}
as desired.
\end{proof}

We consider the class $B1(5,n)$, where we have $B_1(5,n)$ for odd $n$ and $B^1(5,n)$ for even $n$.  These actually give 2-tangles, and so our notation is $bt_n$ for the Kauffman bracket polynomials.  
  Although this alternating class may appear strange, it follows naturally from the recursion, but we consider these two cases separately.

\begin{lemma}
\label{prop:B15b}
The Kauffman bracket polynomials of $B1(5,n)$ (for indeterminate crossing information) are:
\begin{align}
bt_{n\text{ odd}}= {} &  (S,\widetilde{N}^\frac{n-3}{2})+\sum_{i=0}^\frac{n-3}{2}(h_{n-1-2i},X,\widetilde{N}^i) +\sum_{i=0}^\frac{n-5}{2}(h_{n-2-2i},\widetilde{R},\widetilde{N}^i), \\
bt_{n\text{ even}}= {} & (g_2,N^\frac{n-2}{2}) +\sum_{i=0}^\frac{n-4}{2}(h_{n-1-2i},X,N^i) +\sum_{i=0}^\frac{n-4}{2}(h_{n-2-2i},R,N^i),
\end{align}
for $n\geq3$ and with base cases $bt_1=1$ and $bt_2=(g_2)$, where $\square^i$ represents $[\square,\square,\ldots,\square]$, written $i$ times, and where  
$X=\delta[A^\pm,A^\pm]+[A^\pm,A^\mp]+[A^\mp,A^\pm]$, $N=[f_2^\mp,A^\mp,A^\mp,A^\mp]$, $\widetilde{N}=[A^\mp,f_2^\mp,A^\mp,A^\mp]$, $R=[f_2^\mp,A^\pm,A^\mp,A^\mp]$, $\widetilde{R}=[A^\pm,f_2^\mp,A^\mp,A^\mp]$, and $S=[f_2^\pm,f_2^\mp,A^\mp,A^\mp]$.  Recall that $g_2=([X]+\delta[A^\mp,A^\mp])$ is the Kauffman bracket polynomial of $T(4,2)$ from Section \ref{sec:4}.

In particular, here is a list of Kauffman brackets for this class:
\begin{align*}
bt_3= {} & (h_2,X)+(S) \\
bt_4= {} & (h_3,X)+(h_2,R)+(g_2,N) \\
bt_5= {} & (h_4,X)+(h_3,\widetilde{R})+(h_2,X,\widetilde{N})+(S,\widetilde{N}) \\
bt_6= {} & (h_5,X)+(h_4,R)+(h_3,X,N)+(h_2,R,N)+(g_2,N,N) \\
bt_7= {} & (h_6,X)+(h_5,\widetilde{R})+(h_4,X,\widetilde{N})+(h_3,\widetilde{R},\widetilde{N})+(h_2,X,\widetilde{N},\widetilde{N})+(S,\widetilde{N},\widetilde{N}) \\
bt_8= {} & (h_7,X)+(h_6,R)+(h_5,X,N)+(h_4,R,N)+(h_3,X,N,N)+(h_2,R,N,N)+(g_2,N,N,N) \\
bt_9= {} & (h_8,X)+(h_7,\widetilde{R})+(h_6,X,\widetilde{N})+(h_5,\widetilde{R},\widetilde{N})+(h_4,X,\widetilde{N},\widetilde{N})+(h_3,\widetilde{R},\widetilde{N},\widetilde{N}) \\
& +(h_2,X,\widetilde{N},\widetilde{N},\widetilde{N})+(S,\widetilde{N},\widetilde{N},\widetilde{N}) \\
bt_{10}= {} & (h_9,X)+(h_8,R)+(h_7,X,N)+(h_6,R,N)+(h_5,X,N,N)+(h_4,R,N,N) \\
& +(h_3,X,N,N,N)+(h_2,R,N,N,N)+(g_2,N,N,N,N)
\end{align*}
\end{lemma}

\begin{proof}
The case $B_1(5,1)$ is equivalent to $T(4,1)$, and so $bt_1=1$.  The case $B^1(5,2)$ is equivalent to $T(4,2)$, and so we get $bt_2=g_2$.    The actual base cases $B_1(5,3)$ and $B^1(5,4)$ for the induction are handled in Examples \ref{ex:BUp253} and \ref{ex:BUp254}.

Prior to addressing the induction hypothesis, consider the four resolutions on the last two crossings.  Three of them result in a rectangular billiard table of width one less, and the other results in a bumpered billiard table of width one less with three squares removed from the last column.  This last table itself has two resolutions yielding in a rectangular billiard table and a bumpered billiard table, both of width two less than the original.  This is summarized in the following equations:
\begin{align}
bt_{n\text{ odd}}= {} & (h_{n-1},X)+(h_{n-2},A^\pm,f_2^\mp,A^\mp,A^\mp)+(bt_{n-2},A^\mp,f_2^\mp,A^\mp,A^\mp) \text{ and} \\
bt_{n\text{ even}}= {} & (h_{n-1},X)+(h_{n-2},f_2^\mp,A^\pm,A^\mp,A^\mp)+(bt_{n-2},f_2^\mp,A^\mp,A^\mp,A^\mp),
\end{align}
where $X=1-(A^{\pm})^4$ if the crossings are both $\pm$ or else $X=0$, as in Equation \ref{eq:X}.

Suppose the induction hypothesis holds for odd $n$.  Then for odd $n+2$,
\begin{align*}
b_{n+2\text{ odd}}= {} & (h_{n+1},X)+(h_{n},A^\pm,f_2^\mp,A^\mp,A^\mp)+(bt_{n},A^\mp,f_2^\mp,A^\mp,A^\mp) \\
= {} & (h_{n+1},X)+(h_{n},\widetilde{R})+ (S,\widetilde{N}^\frac{n-1}{2}) +\sum_{i=0}^\frac{n-3}{2}(h_{n-1-2i},X,\widetilde{N}^{i+1}) +\sum_{i=0}^\frac{n-5}{2}(h_{n-2-2i},\widetilde{R},\widetilde{N}^{i+1}) \\
= {} &  (S,\widetilde{N}^\frac{n-1}{2}) + \sum_{i=0}^\frac{n-1}{2}(h_{n+1-2i},X,\widetilde{N}^{i}) +\sum_{i=0}^\frac{n-3}{2}(h_{n-2i},\widetilde{R},\widetilde{N}^{i}),
\end{align*}
as desired.

Suppose the induction hypothesis holds for even $n$.  Then for even $n+2$,
\begin{align*}
b_{n+2\text{ even}}= {} & (h_{n-1},X)+(h_{n-2},f_2^\mp,A^\pm,A^\mp,A^\mp)+(bt_{n-2},f_2^\mp,A^\mp,A^\mp,A^\mp) \\
= {} &(h_{n+1},X)+(h_{n},R)+(g_2,N^\frac{n}{2}) +\sum_{i=0}^\frac{n-4}{2}(h_{n-1-2i},X,N^{i+1}) +\sum_{i=0}^\frac{n-4}{2}(h_{n-2-2i},R,N^{i+1})\\
= {} & (g_2,N^\frac{n}{2}) +\sum_{i=0}^\frac{n-2}{2}(h_{n+1-2i},X,N^{i}) +\sum_{i=0}^\frac{n-2}{2}(h_{n-2i},R,N^{i}),
\end{align*}
as desired.
\end{proof}

\begin{remark}
\label{rem:xis0}
Note that some of these terms might disappear because $X=0$ when its two crossings are opposite of each other, as in Equation \ref{eq:X}.
\end{remark}

\begin{question}  This leads to the following open questions:

\begin{enumerate}
	\item Are there natural classes of knots and links for which we have several pairs of opposite crossings, leading to easier formulae?
	\item Instead of a recursion built from the first pair of crossings, could one construct a recursion for a given (more central) pair of crossings, so that there would be more cancellation if these crossings were opposite of each other?
\end{enumerate}
\end{question}

\subsection{Proof of the main theorem}
\label{subsec:proofs}
We conclude this section with the proof.

\begin{proof}
The first few base cases are handled in Examples \ref{ex:T51-2}, \ref{ex:T53}, \ref{ex:T54}, and \ref{ex:T55}.

Prior to addressing the induction hypothesis, consider first the four resolutions on the last two crossings.  One of them results in a rectangular billiard table of width one less, another results in two nugatory crossing that, when resolved, achieves a rectangular billiard table of width two less, and the other two result in bumpered billiard tables of width one less.  This is summarized in the following equations:
\begin{align}
h_{n\text{ odd}}= {} & (h_{n-1},A^\pm,A^\pm)+(h_{n-2},K)+(bt_{n-1},A^\pm,A^\mp)+(b_{n-1},f_2^\mp,\_,A^\mp,A^\pm) \text{ and} \\
h_{n\text{ even}}= {} & (h_{n-1},A^\pm,A^\pm)+(h_{n-2},K)+(bt_{n-1},A^\mp,A^\pm)+(b_{n-1},f_2^\mp,A^\pm,A^\mp),
\end{align}
where for odd $n$, the $f_2^\mp$ is ordered \emph{before the last term} of the $b_{n-1}$ due to placement of the last crossing within the billiard table.

Suppose first that the induction hypothesis holds for odd $n$, for even $n-1$, and for all indexes less than this.  Then for even $n+1$,
\begin{align*}
h_{n+1\text{ even}}= {} & (h_{n},A^\pm,A^\pm)+(h_{n-1},K)+(bt_{n},A^\mp,A^\pm)+(b_{n},f_2^\mp,A^\pm,A^\mp) \\
= {} & (h_{n},A^\pm,A^\pm)+(h_{n-1},K)\\
& +(S,\widetilde{N}^\frac{n-3}{2},A^\mp,A^\pm)+\sum_{i=0}^\frac{n-3}{2}(h_{n-1-2i},X,\widetilde{N}^i,A^\mp,A^\pm) +\sum_{i=0}^\frac{n-5}{2}(h_{n-2-2i},\widetilde{R},\widetilde{N}^i,A^\mp,A^\pm) \\
& +(M,K^{\frac{n-3}{2}},L) +\sum_{i=0}^\frac{n-3}{2}(h_{n-1-2i},A^\pm,K^i,L) +\sum_{i=0}^\frac{n-5}{2}(h_{n-2-2i},L,K^i,L) \\
= {} & (h_{n},A^\pm,A^\pm)+(h_{n-1},K)+(S,\widetilde{N}^\frac{n-3}{2},A^\mp,A^\pm)+(M,K^{\frac{n-3}{2}},L) \\
& +\sum_{i=0}^\frac{n-3}{2}(h_{n-1-2i},[X,\widetilde{N}^i,A^\mp,A^\pm]+[A^\pm,K^i,L]) +\sum_{i=0}^\frac{n-5}{2}(h_{n-2-2i},[\widetilde{R},\widetilde{N}^i,A^\mp,A^\pm]+[L,K^i,L])\\
= {} & (h_{n},P_1)+(h_{n-1},K)+(Q_{n}) \\
&  +\sum_{i=1}^\frac{n-3}{2}(h_{n-1-2i},P'_{2i+2}) + (h_{n-1},[X,A^\mp,A^\pm]+[A^\pm,L]) +\sum_{i=0}^\frac{n-5}{2}(h_{n-2-2i},\widetilde{P}'_{2i+3})\\
\end{align*}
\begin{align*}
= {} & (Q_{n}) +\sum_{j=2}^{n}(h_{j},P_{n+1-j})\\
= {} & (Q_{n}) +(h_{2},P_{n-1}) + (h_{3},P_{n-2}) +\sum_{j=4}^{n}(h_{j},P_{n+1-j})\\
= {} & (Q_{n}) +(h_{2},P_{n-1}) + (h_{3},P_{n-2}) +\sum_{j=4}^{n}\sum_{i=3}^{j-1}([h_3,P_{i-2}]+[h_2,P_{i-1}]+[Q_i],\mathcal{P}_{j-1-i},P_{n+1-j})\\
= {} & (h_{3},P_{n-2}) +(h_{2},P_{n-1}) + (Q_{n}) +\sum_{i=3}^{n-1}([h_3,P_{i-2}]+[h_2,P_{i-1}]+[Q_i],\mathcal{P}_{n-i})\\
= {} & \sum_{i=3}^{n}([h_3,P_{i-2}]+[h_2,P_{i-1}]+[Q_i],\mathcal{P}_{n-i}) = h_{n+1},
\end{align*}
as desired.

Supposing on the other hand that the induction hypothesis holds for even $n$, we obtain a very similar proof for odd $n+1$, omitted here.
\end{proof}

\section{Determining writhe}
\label{sec:writhe}

Often the coefficients of the Kauffman bracket polynomial are enough to determine the Jones polynomial.  We include the following properties, which can be easily checked, without their proofs.

\begin{property}
\label{prop:3writhe}
Consider a knot $T(3,b)$.  
  Then the writhe of $T(3,b+3)$ is given by:
\begin{equation*}
w(T(3,b+3))=
\begin{cases}
w(T(3,b))+(\pm1)+(\mp1)+(\pm1) &\mbox{ when $b\equiv 1 \mod 3$ and}\\
w(T(3,b))+(\pm1)+(\pm1)+(\mp1) &\mbox{ when $b\equiv 2 \mod 3$,}
\end{cases}
\end{equation*}
where the $\pm$ correspond to the signs of the three additional (ordered) crossings.
\end{property}

\begin{property}
\label{prop:5writhe}
Consider a knot $T(5,b)$.  
  Then the writhe of $T(5,b+5)$ is given by $w(T(5,b+5))=$
\begin{equation*}
\begin{cases}
w(T(5,b))+(\pm1)+(\pm1)+(\mp1)+(\pm1) & \\
\hspace{5mm} +(\mp1)+(\mp1)+(\mp1)+(\pm1)+(\pm1)+(\pm1) &\mbox{ when $b\equiv 1 \mod 5$ and $2|b$, }\\
w(T(5,b))+(\pm1)+(\pm1)+(\pm1)+(\mp1) & \\
\hspace{5mm} +(\mp1)+(\mp1)+(\pm1)+(\mp1)+(\pm1)+(\pm1) &\mbox{ when $b\equiv 1 \mod 5$ and $2 \nmid b$, }\\
w(T(5,b))+(\pm1)+(\pm1)+(\pm1)+(\mp1) & \\
\hspace{5mm} +(\mp1)+(\pm1)+(\pm1)+(\pm1)+(\mp1)+(\mp1) &\mbox{ when $b\equiv 2 \mod 5$ and $2| b$, }\\
w(T(5,b))+(\pm1)+(\pm1)+(\mp1)+(\pm1) & \\
\hspace{5mm} +(\pm1)+(\mp1)+(\pm1)+(\pm1)+(\mp1)+(\mp1) &\mbox{ when $b\equiv 2 \mod 5$ and $2 \nmid b$, }\\
w(T(5,b))+(\pm1)+(\pm1)+(\mp1)+(\mp1) & \\
\hspace{5mm} +(\pm1)+(\pm1)+(\pm1)+(\mp1)+(\mp1)+(\pm1) &\mbox{ when $b\equiv 3 \mod 5$ and $2| b$, }\\
w(T(5,b))+(\pm1)+(\pm1)+(\mp1)+(\mp1) & \\
\hspace{5mm} +(\pm1)+(\pm1)+(\mp1)+(\pm1)+(\pm1)+(\mp1) &\mbox{ when $b\equiv 3 \mod 5$ and $2 \nmid b$, }\\
w(T(5,b))+(\pm1)+(\pm1)+(\pm1)+(\pm1) & \\
\hspace{5mm} +(\pm1)+(\mp1)+(\mp1)+(\mp1)+(\pm1)+(\mp1) &\mbox{ when $b\equiv 4 \mod 5$ and $2| b$, }\\
w(T(5,b))+(\pm1)+(\pm1)+(\pm1)+(\pm1) & \\
\hspace{5mm} +(\mp1)+(\pm1)+(\mp1)+(\mp1)+(\mp1)+(\pm1) &\mbox{ when $b\equiv 4 \mod 5$ and $2 \nmid b$, }
\end{cases}
\end{equation*}
where the $\pm$ correspond to the signs of the ten additional (ordered) crossings.
\end{property}

Observe that the cases of even and odd $b$ within each equivalence class simply interchange the terms within each pair, similar to other behavior earlier in the paper like $N$ and $\widetilde{N}$, $R$ and $\widetilde{R}$.

\newcommand{\etalchar}[1]{$^{#1}$}
\def\cprime{$'$}
\providecommand{\bysame}{\leavevmode\hbox to3em{\hrulefill}\thinspace}
\providecommand{\MR}{\relax\ifhmode\unskip\space\fi MR }
\providecommand{\MRhref}[2]{%
  \href{http://www.ams.org/mathscinet-getitem?mr=#1}{#2}
}
\providecommand{\href}[2]{#2}


\begin{thebibliography}{DFKLS08}

\bibitem[BDHZ09]{BDHZ}
Adam Boocher, Jay Daigle, Jim Hoste, and Wenjing Zheng, \emph{Sampling
  {L}issajous and {F}ourier knots}, Experiment. Math. \textbf{18} (2009),
  no.~4, 481--497.

\bibitem[BEJS10]{BEJS}
Arthur~T. {Benjamin}, Larry {Ericksen}, Pallavi {Jayawant}, and Mark
  {Shattuck}, \emph{{Combinatorial trigonometry with Chebyshev polynomials}},
  {J. Stat. Plann. Inference} \textbf{140} (2010), no.~8, 2157--2160.

\bibitem[BHJS94]{BHJS}
M.~G.~V. Bogle, J.~E. Hearst, V.~F.~R. Jones, and L.~Stoilov, \emph{Lissajous
  knots}, J. Knot Theory Ramifications \textbf{3} (1994), no.~2, 121--140.

\bibitem[BHS09]{BHS}
Matthias {Beck}, Christian {Haase}, and Steven~V {Sam}, \emph{{Grid graphs,
  Gorenstein polytopes, and domino stackings}}, {Graphs Comb.} \textbf{25}
  (2009), no.~4, 409--426.

\bibitem[{Bir}85]{Bir:3braid}
Joan~S. {Birman}, \emph{{On the Jones polynomial of closed 3-braids}},
  {Invent. Math.} \textbf{81} (1985), 287--294.

\bibitem[BNMea]{knotatlas}
Dror Bar-Natan, Scott Morrison, and et~al, \emph{The {K}not {A}tlas},
  http://katlas.org.

\bibitem[BW09]{BenWal}
Arthur~T. {Benjamin} and Daniel {Walton}, \emph{{Counting on Chebyshev
  polynomials}}, {Math. Mag.} \textbf{82} (2009), no.~2, 117--126.

\bibitem[CDR14]{CoDaRu}
Moshe {Cohen}, Oliver~T. {Dasbach}, and Heather~M. {Russell}, \emph{{A twisted
  dimer model for knots}}, {Fundam. Math.} \textbf{225} (2014), no.~1, 57--74.

\bibitem[{Coh}12]{Co:jones}
Moshe {Cohen}, \emph{{A determinant formula for the Jones polynomial of pretzel
  knots}}, {J. Knot Theory Ramifications} \textbf{21} (2012), no.~6, 1250062,
  23.

\bibitem[Com97]{Comstock}
Elting~H. Comstock, \emph{The real singularities of harmonic curves of three
  frequencies}, Transactions of the Wisconsin Academy of Sciences, Arts and
  Letters \textbf{XI} (1896-1897), 452--464.

\bibitem[Cro04]{Crom}
Peter~R. Cromwell, \emph{Knots and links}, Cambridge University Press,
  Cambridge, 2004.

\bibitem[CT12]{Co:clock}
Moshe Cohen and Mina Teicher, \emph{Kauffman's clock lattice as a graph of
  perfect matchings: a formula for its height}, arXiv:1211.2558, 2012.

\bibitem[{deW}97]{deW}
Benjamin~M.M. {de Weger}, \emph{{Padua and Pisa are exponentially far apart}},
  {Publ. Mat., Barc.} \textbf{41} (1997), no.~2, 631--651.

\bibitem[DFKLS08]{DaFuKaLiSt}
Oliver~T. Dasbach, David Futer, Efstratia Kalfagianni, Xiao-Song Lin, and
  Neal~W. Stoltzfus, \emph{The {J}ones polynomial and graphs on surfaces}, J.
  Combin. Theory Ser. B \textbf{98} (2008), no.~2, 384--399.

\bibitem[DS10]{DuzShk}
Sergei Duzhin and Mikhail Shkolnikov, \emph{A formula for the {HOMFLY}
  polynomial of rational links}, arXiv:1009.1800, 2010.

\bibitem[{Eul}05]{Eul:grid}
Reinhardt {Euler}, \emph{{The Fibonacci number of a grid graph and a new class
  of integer sequences}}, {J. Integer Seq.} \textbf{8} (2005), no.~2, art.
  05.2.6, 16.


\bibitem[GJ14]{GuJo}
Jie Gu and Hans Jockers, \emph{A note on colored {HOMFLY} polynomials for
  hyperbolic knots from {WZW} models}, arXiv:1407.5643v1, 2014.

\bibitem[HZ07]{HosZir}
Jim Hoste and Laura Zirbel, \emph{Lissajous knots and knots with {L}issajous
  projections}, Kobe J. Math. \textbf{24} (2007), no.~2, 87--106.

\bibitem[JP98]{JonPrz}
Vaughan F.~R. Jones and J{\'o}zef~H. Przytycki, \emph{Lissajous knots and
  billiard knots}, Knot theory ({W}arsaw, 1995), Banach Center Publ., vol.~42,
  Polish Acad. Sci. Inst. Math., Warsaw, 1998, pp.~145--163.

\bibitem[{Kan}89]{Kan2}
Taizo {Kanenobu}, \emph{{Examples on polynomial invariants of knots and links
  II}}, {Osaka J. Math.} \textbf{26} (1989), no.~3, 465--482.

\bibitem[{Kan}90]{Kan1}
\bysame, \emph{{Jones and Q polynomials for 2-bridge knots and links}}, {Proc.
  Am. Math. Soc.} \textbf{110} (1990), no.~3, 835--841.

\bibitem[Kau87]{KauffmanBracket}
Louis~H. Kauffman, \emph{State models and the {J}ones polynomial}, Topology
  \textbf{26} (1987), no.~3, 395--407.

\bibitem[Ken09]{Ken}
Richard Kenyon, \emph{Lectures on dimers}, Statistical mechanics, IAS/Park City
  Math. Ser., vol.~16, Amer. Math. Soc., Providence, RI, 2009, pp.~191--230.

\bibitem[KP80]{KlaPol}
David {Klarner} and Jordan {Pollack}, \emph{{Domino tilings of rectangles with
  fixed width}}, {Discrete Math.} \textbf{32} (1980), 45--52.

\bibitem[KP10]{PVK:fibknots}
Pierre-Vincent {Koseleff and} Daniel {Pecker}, \emph{{On Fibonacci knots}}, {Fibonacci Q.}
  \textbf{48} (2010), no.~2, 137--143.

\bibitem[KP11a]{KosPec4}
\bysame, \emph{Chebyshev diagrams for two-bridge knots},
  Geom. Dedicata \textbf{150} (2011), 405--425.

\bibitem[KP11b]{KosPec3}
\bysame, \emph{Chebyshev knots}, J. Knot Theory Ramifications \textbf{20}
  (2011), no.~4, 575--593.

\bibitem[KP12]{KosPec:Harm}
\bysame, \emph{Harmonic Knots}, arXiv:1203.4376, 2012.

\bibitem[KP15]{PVK:fib}
\bysame, \emph{On {A}lexander-{C}onway
  polynomials of two-bridge links}, J. Symb. Comput. (2015).

\bibitem[Lam97]{Lam}
Christoph Lamm, \emph{There are infinitely many {L}issajous knots}, Manuscripta
  Math. \textbf{93} (1997), no.~1, 29--37.

\bibitem[Lam99]{Lam:dis}
\bysame, \emph{Zylinder-{K}noten und symmetrische {V}ereinigungen}, Bonner
  Mathematische Schriften [Bonn Mathematical Publications], 321, Universit\"at
  Bonn Mathematisches Institut, Bonn, 1999, Dissertation, Rheinische
  Friedrich-Wilhelms-Universit{\"a}t Bonn, Bonn, 1999.

\bibitem[Lew14]{Lew:dis}
Sam Lewallen, \emph{Floerg{\r{a}}sbord}, arXiv:407.0769v1, 2014.

\bibitem[LLS09]{LeeLeeSeo}
Eunju {Lee}, Sang~Youl {Lee}, and Myoungsoo {Seo}, \emph{{A recursive formula
  for the Jones polynomial of 2-bridge links and applications}}, {J. Korean
  Math. Soc.} \textbf{46} (2009), no.~5, 919--947.

\bibitem[LM87]{LicMil}
W.B.R. {Lickorish} and Kenneth~C. {Millett}, \emph{{A polynomial invariant of
  oriented links}}, {Topology} \textbf{26} (1987), 107--141.

\bibitem[LP86]{LovPlum}
L.~Lov{\'a}sz and M.~D. Plummer, \emph{Matching theory}, North-Holland
  Mathematics Studies, vol. 121, North-Holland Publishing Co., Amsterdam, 1986,
  Annals of Discrete Mathematics, 29.

\bibitem[LZ10]{LuZho}
Bin {Lu} and Jianyuan~K. {Zhong}, \emph{{An algorithm to compute the Kauffman
  polynomial of 2-bridge knots}}, {Rocky Mt. J. Math.} \textbf{40} (2010),
  no.~3, 977--993.



\bibitem[{Mas}06]{Masur:erg}
Howard {Masur}, \emph{{Ergodic theory of translation surfaces}}, {Handbook of
  dynamical systems. Volume 1B}, Amsterdam: Elsevier, 2006, pp.~527--547.


\bibitem[MT02]{MasTab:bill}
Howard {Masur} and Serge {Tabachnikov}, \emph{{Rational billiards and flat
  structures}}, {Handbook of dynamical systems. Volume 1A}, Amsterdam:
  North-Holland, 2002, pp.~1015--1089.



\bibitem[{Nak}00]{Nak1}
Shigekazu {Nakabo}, \emph{{Formulas on the HOMFLY and Jones polynomials of
  2-bridge knots and links}}, {Kobe J. Math.} \textbf{17} (2000), no.~2,
  131--144.

\bibitem[{Nak}02]{Nak2}
\bysame, \emph{{Explicit description of the HOMFLY polynomials for 2-bridge
  knots and links}}, {J. Knot Theory Ramifications} \textbf{11} (2002), no.~4,
  565--574.

\bibitem[OEI]{OEIS931}
\emph{The {O}n-{L}ine {E}ncyclopedia of {I}nteger {S}equences},
  http://oeis.org, 2014, Sequence A000931.

\bibitem[QYAQ14]{QYAQ}
Khaled Qazaqzeh, Moh'd Yasein, and Majdoleen Abu-Qamar, \emph{The {J}ones
  polynomial of rational links}, arXiv:1406.4339, 2014.

\bibitem[{Rea}80]{Rea}
Ronald~C. {Read}, \emph{{A note on tiling rectangles with dominoes}},
  {Fibonacci Q.} \textbf{18} (1980), 24--27.

\bibitem[{Sta}85]{Sta:rect}
Richard~P. {Stanley}, \emph{{On dimer coverings of rectangles of fixed
  width}}, {Discrete Appl. Math.} \textbf{12} (1985), 81--87.

\bibitem[{Sto}00]{Sto}
Alexander {Stoimenow}, \emph{{Rational knots and a theorem of Kanenobu}},
  {Exp. Math.} \textbf{9} (2000), no.~3, 473--478.

\bibitem[Thi87]{Th}
Morwen~B. Thistlethwaite, \emph{A spanning tree expansion of the {J}ones
  polynomial}, Topology \textbf{26} (1987), no.~3, 297--309.

\end{thebibliography}
\end{document}